\documentclass[11pt]{article}

\usepackage{mathptmx}

\usepackage[utf8]{inputenc}
\usepackage{amssymb,amsmath,mathdots,amsthm}
\usepackage{enumerate}

\usepackage[a4paper,inner=2cm, outer=3cm,bottom=3cm,top=2cm,marginparsep=1em,marginparwidth=2cm]{geometry}

\usepackage{tikz}
\usetikzlibrary{calc}
\usetikzlibrary{intersections}
\usetikzlibrary {decorations.pathreplacing}

\usepackage[
        colorlinks,
        pdfauthor={Petr Ambrož, Zuzana Masáková, Edita Pelantová}, 
        pdftitle={Geometric and combinatorial properties of spectra of quadratic Pisot numbers},
]{hyperref}

\usepackage[
   backend=bibtex,  
   sorting=nyt,  
   giveninits=true,  
   isbn=false  
]{biblatex}

\DeclareFieldFormat
  [article,inbook,incollection,inproceedings,patent,thesis,unpublished]
  {title}{{\textit{#1}\isdot}}

\DeclareFieldFormat{journaltitle}{#1}
\DeclareFieldFormat{booktitle}{#1}

\DeclareFieldFormat{pages}{#1}

\renewbibmacro{in:}{}

\DeclareFieldFormat{url}{%
  \iffieldundef{doi}{\mkbibacro{URL}\addcolon\space\url{#1}}{}%
}

\newcommand{\floor}[1]{\left\lfloor{#1}\right\rfloor}
\newcommand{\lowvdots}{\vphantom{\int\limits^x}\smash{\vdots}}

\newcommand{\bu}{\boldsymbol{u}}
\newcommand{\Parikh}[1][\kern0.3ex]{\mathrm{P}_{\kern-0.3ex#1}}

\newcommand{\N}{\mathbb{N}}
\newcommand{\Z}{\mathbb{Z}}

\newcommand{\R}{\mathbb{R}}
\newcommand{\A}{\mathcal{A}}
\newcommand{\B}{\mathcal{B}}
\newcommand{\D}{\mathcal{D}}
\renewcommand{\L}{\mathcal{L}}

\newtheorem{theorem}{Theorem}
\newtheorem{proposition}[theorem]{Proposition}
\newtheorem{corollary}[theorem]{Corollary}
\newtheorem{lemma}[theorem]{Lemma}

\theoremstyle{definition}

\newtheorem{definition}[theorem]{Definition}
\newtheorem{example}[theorem]{Example}
\newtheorem{remark}[theorem]{Remark}

\newcommand\blfootnote[1]{%
  \begingroup
  \renewcommand\thefootnote{}\footnote{#1}%
  \addtocounter{footnote}{-1}%
  \endgroup
}

\begin{document}

\title{Lattice Bounded Distance Equivalence for 1D Delone Sets with Finite Local Complexity}

\author{Petr Ambrož \quad Zuzana Masáková \quad Edita Pelantová\\[2mm]
\small Department of Mathematics, Faculty of Nuclear Sciences and Physical Engineering\\
\small Czech Technical University in Prague, Trojanova 13, 120 00 Praha, Czech Republic
\vspace{-\baselineskip}}







\date{}


\allowdisplaybreaks

\maketitle

\blfootnote{Emails: \texttt{petr.ambroz@fjfi.cvut.cz}, \texttt{zuzana.masakova@fjfi.cvut.cz}, \texttt{edita.pelantova@fjfi.cvut.cz}}

\begin{abstract}
  Spectra of suitably chosen Pisot-Vijayaraghavan numbers represent non-trivial examples of self-similar Delone
  point sets of finite local complexity, indispensable in quasicrystal modeling.  For the case of
  quadratic Pisot units we characterize, dependingly on digits in the corresponding numeration
  systems, the spectra which are bounded distance to an average lattice. Our method stems in
  interpretation of the spectra in the frame of the cut-and-project method.  Such structures are
  coded by an infinite word over a finite alphabet which enables us to exploit combinatorial notions
  such as balancedness, substitutions and the spectrum of associated incidence matrices.



\end{abstract}


\section{Introduction}


Point sets arising by a small fluctuation of elements of a lattice in $\mathbb{R}^d$ are said to be
bounded distance equivalent to a lattice or shortly BDL sets.  Such sets are studied both in
mathematical and physical literature. In physics, the so-called `average lattice' helps to
understand the diffraction properties of a BDL set which models the atomic sites of a
non-crystallographic material~\cite{Baranidharan}. Other properties of the material
may be influenced by the parameters of the average lattice. For example, average spacing between
sites in the one-dimensional Fibonacci chain determines the sound velocity, see~\cite{NauWaThoBar}.

In order to enable such a study for a larger class of aperiodic structures, we propose to focus on
one-dimensional sequences arising as the spectra of real numbers in the sense introduced by Erd\H os
et al.~\cite{Erdos}. The Fibonacci sequence also can be interpreted as the spectrum of the golden
ratio. This algebraic concept is intrinsically connected with numeration systems in irrational
bases, which allow efficient algorithms for computation on quasicrystals~\cite{BuFrGaKr}. Similarly,
as it is the case of the Fibonacci chain, the spectra of suitably chosen Pisot-Vijayaraghavan bases
arise by the cut-and-project scheme. Starting with the break-through paper of Kramer and
Neri~\cite{KramerNeri}, the cut and projection is considered as an irreplaceable tool in
quasicrystallography.

We start our study by considering general one-dimensional sets that can be expressed using a
strictly increasing sequence of points $\Lambda =\{x_n\}_{n\in \mathbb{Z}}$ for which it moreover
holds that the distances $x_{n+1}-x_n$ between consecutive elements take only finitely many
values. This property is said finite local complexity and the sequence of gaps can be coded by a
bidirectional infinite word over a finite alphabet.  Decision whether a point set with finite local
complexity is bounded distance equivalent to a lattice is a difficult problem in general.  When the
infinite word coding the set $\Lambda$ is balanced, the task is easy.

On the other hand, associating positive lengths to the letters of a finite alphabet $\A$, every
bidirectional infinite word over $\A$ has a geometric representation $\{x_n\}_{n\in \mathbb{Z}}$,
which is of finite local complexity. We study infinite words generated using rewriting rules over
the alphabet $\mathcal{A}$, i.e., fixed points of morphisms over $\A$. If the infinite word is
balanced, the geometric representation is obviously BDL. The balance property can be decided using
the result of Adamczewski~\cite{AdamczBalance}.  The balance property is nevertheless not a
necessary condition for lattice bounded distance equivalence.  In Proposition~\ref{prop:BDL_repr_of_fixed_point}
we give a weaker sufficient condition so that a fixed point of a
morphism has a non-trivial geometric representation with the BDL property.  Similarly as in the
result of Adamczewski, our sufficient condition is formulated using the spectrum of the incidence
matrix of the morphism.

A class of one-dimensional sets studied for the BDL property is the family of the sets $\Z_{\beta}$
for $\beta>1$, where $\Z_\beta$ is formed by the so-called $\beta$-integers, as defined
in~\cite{BuFrGaKr}. In~\cite{Gazeau}, Gazeau et al.\ have shown that $\Z_\beta$ is BDL if $\beta$ is
a Pisot number. (Recall that $\alpha$ is a Pisot number if it is an algebraic integer greater than 1
with all algebraic conjugates in the interior of the unit disc.)

A different type of a one-dimensional set with finite local complexity is the generalized Erd\H os
spectrum of a real base $\alpha$, $|\alpha|>1$, namely the set
\begin{equation}\label{eq:defspectrum}
 X^{\mathcal{D}}(\alpha) =
  \Big\{\sum_{i=0}^na_i\alpha^i : n\in\N, a_i\in\mathcal{D}\Big\}.
\end{equation}
where $\mathcal{D}$ is a finite set of consecutive integers containing 0 with $\#\D>\beta$ and
$|\alpha|$ is a Pisot number. Directly from the definition~\eqref{eq:defspectrum} we see that the spectrum is self-similar with factor $\alpha$, 
i.e. $\alpha X^{\mathcal{D}}(\alpha) \subset X^{\mathcal{D}}(\alpha)$.

For the so-called confluent Parry numbers (which belong to the family of Pisot numbers) the spectrum
with alphabet $\D=\{0,\dots,\floor{\beta}\}$ coincides with
$\Z_\beta\cap\R^{\geq0}$~\cite{Frougny,GaHa}, and thus it can be symmetrically extended to a BDL
set. Other cases of the spectra have not yet been inspected for the BDL property.  We study the
question for a family of the spectra of quadratic Pisot units with the alphabet being any set of 
consecutive integers ensuring the Delone property of the spectrum.

In the case that $|\alpha|$ is a quadratic Pisot unit, we prove in Theorem~\ref{t:spectra}, that the
BDL property depends only on the cardinality of the alphabet $\mathcal{D}$, and, surprisingly, the
values $\#\mathcal{D}$, for which $X^\mathcal{D}(\alpha)$ is BDL form an arithmetic sequence. The
main tool used in the proof is the result of~\cite{MaPaPe} which interprets the spectrum
$X^\mathcal{D}(\alpha)$ as a cut-and project set. This interpretation further allows us to 
apply classical theorem of Kesten~\cite{Kesten}.

The paper is organized as follows. In Section~\ref{sec:balance} we recall necessary notions of
combinatorics on words, in particular properties of balanced words. We derive that the image of a
balanced word by a non-erasing morphism is also balanced. In Section~\ref{sec:BDL1d} we formalize
the notion of bounded distance equivalence to a lattice and we study this property for geometric
representations of infinite words (Section~\ref{sec:geomrep}).  For fixed points of morphisms, we
put into context their balancedness (Section~\ref{sec:subst}) and their BDL property
(Section~\ref{sec:BDLsubst}).  Definition of a cut-and-project set and a criterion for its BDL
property are given in Section~\ref{sec:BDLcap}.  Section~\ref{sec:spectrum} is devoted to the study
of spectra of quadratic Pisot numbers and their BDL property.


\section{Balanced infinite words}\label{sec:balance}


In this chapter we recall basic notions connected to infinite words and the property of balancedness
which plays an important role in our considerations.

Let $\A$ be a finite alphabet. The set of finite words over $\A$, equipped with the operation of
concatenation and the empty word $\epsilon$ as the neutral element, is a monoid, which we denote by
$\A^*$.  We will also consider infinite words, namely one-directional infinite words
$\bu=u_0u_1u_2\cdots\in\A^\N$ and bi-directional infinite words $\bu=\cdots
u_{-2}u_{-1}|u_0u_1u_2\cdots\in\A^\Z$. The delimiter $|$ is important when deciding whether two
bi-directional infinite words coincide. If a word $u$ (finite or infinite) is written as $u=vwv'$
for some (possibly empty) words $v,w,v'$, then $v$ is a prefix, $v'$ a suffix and $w$ a factor of
$u$. In particular, we denote by $u_{[i,j)}=u_iu_{i+1}\cdots u_{j-1}$.  The length of a finite word
$w=w_1\cdots w_n$ is denoted by $|w|=n$. The number of letters $a\in\A$ occurring in the word $w$
is denoted by $|w|_a$. For a finite word $w$ over an alphabet $\A=\{a_1,\dots,a_d\}$ we define its
Parikh vector $\Parikh(w)=(|w|_{a_1},\dots,|w|_{a_d})^{\mathrm{T}}$.  Obviously,
$|w|=(1,1,\ldots,1)\Parikh(w)$.  The set of all factors of an infinite word $\bu$ is the language
of $\bu$ which we denote by $\L(\bu)$.  If there exists a positive integer $p$ such that
$u_{n+p}  = u_n$ for all $n \in \mathbb{Z}$, the word $\bu=\cdots u_{-2}u_{-1}u_0u_1u_2\cdots\in\A^\Z$ is
called periodic, any its factor of length $p$ is its period .
If $u_{n+p} = u_n$ for all $n \in \mathbb{Z}$ up to finitely many exceptions, $\bu$ is eventually
periodic.  Otherwise, $\bu$ is aperiodic.

The combinatorial property in focus is balancedness of the infinite word $\bu$.

\begin{definition}
  An infinite word $\bu\in\A^\N$ ($\A^\Z$) is said to be balanced, if there exists a constant $c>0$
  such that for any pair of factors $w,v$ of $\bu$ of the same length $|w|=|v|$ and for any letter
  $a\in\A$ we have $|w|_a-|v|_a\leq c$.
\end{definition}

Note that sometimes, the infinite word with the above property is called $c$-balanced and the notion
of balanced words is taken only for $c=1$. Binary aperiodic words which are $1$-balanced are called
sturmian. They were introduced by Morse and Hedlund in~\cite{MoHe} and thanks to an overwhelming
number of equivalent definitions they admit, they are the most studied object in the combinatorics
on infinite words~\cite{berstel}.  $1$-balanced words over a multiletter alphabet were described
in~\cite{graham-jct-15} and~\cite{hubert-tcs-242}.

It is easy to show that the frequencies of letters in a balanced word exist. If $\bu$ is a
bidirectional infinite word over an alphabet $\A$, then the frequency of $a\in\A$ in $\bu$ is
defined by the limit
\[
\varrho_a=\lim_{\substack{w\in\L(\bu)\\ |w|\to\infty}}\frac{|w|_a}{|w|}.
\]
if it exists. For completeness, we include the proof.

\begin{proposition}\label{p:frekvence}
  Let $\bu$ be a bidirectional $c$-balanced infinite word. Then the frequency $\varrho_a$ of any
  letter $a\in\A$ in $\bu$ exists. Moreover, for any factor $w\in\L(\bu)$ and any letter $a\in\A$ we
  have
  \[
  \big||w|_a - \varrho_a|w|\big|\leq c.
  \]
\end{proposition}

\begin{proof}
  Consider a fixed $a\in\A$. Define $A_n$ as the number of occurrences of the letter $a$ in the
  prefix $u_0\cdots u_{n-1}$, i.e., $A_n:=|u_{[0,n)}|_a$. We will show that there exists a constant
  $c$ such that $A_n$ satisfies
  \begin{equation}\label{eq:farkas}
    A_n+A_m - c\leq A_{n+m} \leq A_n+A_m + c,\qquad\text{ for every } m,n\in\N.
  \end{equation}
  Take $n,m\in\N$. Then $u_{[0,m+n)}=u_{[0,n)}w$ for some factor $w$ of $\bu$ of length $|w|=m$. We
  have
  \[
  A_n + A_m - c \leq A_{m+n}=A_n+|w|_a \leq A_n + A_m + c,
  \]
  where $c$ is the constant from the balance property of $\bu$.  Fix $N\in\N$. Then for arbitrary
  $n\in\N$ there exist $k,r\in\N_0$ such that $n=kN+r$, where $k:=\floor{n/N}$, $0\leq r<N$.
  By~\eqref{eq:farkas}, we have
  \[
  kA_N+A_r-kc \leq A_n=A_{kN+r}\leq kA_N+A_r+kc.
  \]
  Dividing the above by $n$ and considering $n\to\infty$ we obtain
  \begin{equation}\label{eq:farkas2}
    \frac{A_N}{N} - \frac{c}{N} \leq \liminf_{n\to\infty} \frac{A_n}{n}\leq \limsup_{n\to\infty}
    \frac{A_n}{n} \leq \frac{A_N}{N} + \frac{c}{N},
  \end{equation}
  where we have used that $\lim_{n\to\infty}\frac{k}{n}=\lim_{n\to\infty}\tfrac1n\floor{n/N} =
  \tfrac{1}{N}$ and $\lim_{n\to\infty}\frac{A_r}{n}=0$.  Since~\eqref{eq:farkas2} holds for every
  $N\in\N$, we obtain that
  \[
  \liminf_{n\to\infty} \frac{A_n}{n}= \limsup_{n\to\infty} \frac{A_n}{n} =\lim_{n\to\infty}\frac{A_n}{n}.
  \]
  We have thus derived that
  $\lim_{n\to\infty}\tfrac{A_n}{n}=\lim_{n\to\infty}\tfrac{1}{n}|u_{[0,n)}|_a$ exists.  For any
  factor $w$ of $\bu$ of length $|w|=n$, we have by the balance property that
  $A_n-c\leq |w|_a\leq A_n +c$, and thus the limit of $\frac{|w|_a}{n}$ exists and is equal to
  \[
  \varrho_a:= \lim_{\substack{w\in\L(\bu)\\ |w|\to\infty}}\frac{|w|_a}{|w|} =
    \lim_{n\to\infty}\frac{A_n}{n}.
  \]
  For a prefix of $\bu$ of length $N$ inequality~\eqref{eq:farkas2} also implies that
  \[
  \big||u_{[0,N)}|_a - \varrho_a|u_{[0,N)}|\big|\leq c
  \]
  for any $N\in\N$. Consider any factor $w\in\L(\bu)$ and an index $i$ such that
  $w=u_iu_{i+1}\cdots u_{i+n-1}$. Then the infinite word $\boldsymbol{w}$ given by $w_n:=u_{n-i}$
  for $n\in\Z$ has the same language as $\bu$ and thus it is $c$-balanced and the
  frequency of a letter $a\in\A$ in $\boldsymbol{w}$ is $\varrho_a$. Since $w$ is a prefix of
  $\boldsymbol{w}$, we obtain that
  \[
  \big||w|_a - \varrho_a|w|\big|\leq c.
  \]
\end{proof}

The balance property of an infinite word is preserved under the action of a morphism. Let $\A,\B$ be
finite alphabets. A mapping $\psi:\A^*\to\B^*$ is a morphism, if $\psi(wv)=\psi(w)\psi(v)$ for any
$w,v\in\A^*$. The morphism is non-erasing, if $\psi(a)$ is not an empty word for every $a\in\A$.
The action of a morphism is naturally extended to infinite words $\bu\in\A^\N$ or $\bu\in\A^\Z$, by
\[
\psi(u_0u_1u_2\cdots)=\psi(u_0)\psi(u_1)\psi(u_2)\cdots,
\]
or
\[
\psi(\cdots u_{-1}|u_0u_1u_2\cdots)=\cdots\psi(u_{-1})|\psi(u_0)\psi(u_1)\psi(u_2)\cdots.
\]
To any morphism $\psi:\A^*\to\B^*$, we associate its incidence matrix $M_\psi$. Its rows and
columns are indexed by $b\in\mathcal{B}$ and $a\in\mathcal{A}$, respectively.  We define
$(M_\psi)_{ba}=|\psi(a)|_b$.  Given a finite word $w\in\A^*$, the Parikh vector $\Parikh(\psi(w))$
of its image $\psi(w)$ can be calculated from the Parikh vector $\Parikh(w)$ of $w$ by
$\Parikh(\psi(w))=M_\psi\Parikh(w)$.

\begin{proposition}\label{p:morphicimagebalance}
  Let $\A,\B$ be finite alphabets.  Let $\bu$ be a balanced bidirectional infinite word
  over $\A$. Let $\psi:\A^*\to\B^*$ be a non-erasing morphism. Then $\psi(\bu)$ is also a
  balanced word.
\end{proposition}

\begin{proof}
  Let $\bu$ be a $c$-balanced infinite word. By Proposition~\ref{p:frekvence}, the frequencies of
  letters in $\bu$ exist and for any factor $w\in\L(\bu)$ we have
  \begin{equation}\label{eq:skorohustotyobec}
    \big| |w|_a-\varrho_a |w|\big| \leq c.
  \end{equation}
  For the morphism $\psi$, denote $\mu:=\max\{|\psi(a)| : a\in\A\}$. Then any factor $v$ in
  $\psi(\bu)$ can be written as $v=p\psi(w)s$, for some $w\in\A^*$, $p,s\in\B^*$, $|p|,|s|<\mu$.  We
  will show that there exist constants $\lambda,\kappa$ and for each letter $b\in\B$ constants
  $\lambda_b,\kappa_b$ (independent of $v$ and $w$) such that for any such $v$ and $w$ we have
  \begin{equation}\label{eq:konstanty}
    \big| |v|-\lambda |w|\big| \leq \kappa \quad\text{ and }\quad
    \big| |v|_b-\lambda_b |w|\big| \leq \kappa_b,\ b\in\B^*.
  \end{equation}

  By Proposition~\ref{p:frekvence}, the frequencies $\varrho_a$ of letters $a\in\A$ in the balanced
  word $\bu$ exist. If $\A=\{a_1,\dots,a_d\}$, denote
  $\varrho_{\bu}=(\varrho_{a_1},\dots,\varrho_{a_d})^{\mathrm{T}}$. Since $v=p\psi(w)s$, we
  can write $|v|=|p|+|s|+(1,1,\dots,1)M_\psi\Parikh(w)$.  Inequality~\eqref{eq:skorohustotyobec} then
  implies that $\Parikh(w)=|w|\varrho_{\bu} + g$, where the vector $g\in\R^d$ has components
  bounded by $c$.  We then have
  \begin{equation}\label{eq:delka}
    |v|=|p|+|s|+(1,1,\dots,1)M_\psi (|w|\varrho_{\bu} + g).
  \end{equation}

  Set
  \[
  \lambda:=(1,1,\dots,1)M_\psi\varrho_{\bu},\quad\text{ and }\quad
  \kappa:= 2\mu + c (1,1,\dots,1) M_\psi
  \left(\begin{smallmatrix} 1 \\ 1 \\ \lowvdots \\ 1 \end{smallmatrix}\right).
  \]
  Note that since $\psi$ is a non-erasing morphism, $(1,\ldots,1)M_\psi$ is a positive
  vector and thus $\lambda> 0$.  Then~\eqref{eq:delka} implies $\big| |v|-\lambda |w|\big| \leq
  \kappa$.  Similarly we derive $\big| |v|_b-\lambda_b |w|\big| \leq \kappa_b$ for any $b\in\B$
  where
  \[
  \lambda_b:=(0,\dots,1,\dots,0)M_\psi\varrho_{\bu},\quad\text{ and }\quad
  \kappa_b:= 2\mu + c (0,\dots,1,\dots,0) M_\psi
  \left(\begin{smallmatrix} 1 \\ 1 \\ \lowvdots \\ 1 \end{smallmatrix}\right).
  \]
  We have thus proved existence of constants so that~\eqref{eq:konstanty} is valid.

  Let now $v^{(1)}$, $v^{(2)}$ be factors of the infinite word $\psi(\bu)$ with
  $|v^{(1)}|=|v^{(2)}|$. Write $v^{(i)}=p^{(i)}\psi(w^{(i)})s^{(i)}$, for some $w^{(i)}\in\A^*$,
  $p^{(i)},s^{(i)}\in\B^*$, $|p^{(i)}|,|s^{(i)}|<\mu$, $i=1,2$.  From~\eqref{eq:konstanty} we have
  \[
  \lambda|w^{(1)}| - \kappa \leq |v^{(1)}|=|v^{(2)}|\leq \lambda|w^{(2)}| + \kappa,
  \]
  which implies
  \[
  \lambda\big(|w^{(1)}|-|w^{(2)}|\big) \leq 2\kappa.
  \]
  For a letter $b\in\B$ we then have using~\eqref{eq:konstanty}
  \[
  \begin{aligned}
    |v^{(1)}|_b-|v^{(2)}|_b &\leq \lambda_b|w^{(1)}|+\kappa_b - \lambda_b|w^{(2)}| + \kappa_b \\
    &\leq \lambda_b \big(|w^{(1)}|-|w^{(2)}|\big) + 2\kappa_b \leq
    \frac{\lambda_b}{\lambda}2\kappa + 2\kappa_b.
  \end{aligned}
  \]
  Setting $C:=\lceil\frac{\lambda_b}{\lambda}2\kappa + 2\kappa_b\rceil$, the above inequality means
  that the infinite word $\psi(\bu)$ is $C$-balanced.
\end{proof}


\section{Bounded distance equivalence}\label{sec:BDL1d}


Let us now formalize the notion of lattice bounded distance equivalence.

\begin{definition}
  A point set $\Lambda\subset\R^d$ is said to be Delone if it is uniformly discrete, that is, there
  exists $r>0$ such that $|x-y|\geq r$ for any $x,y\in\Lambda$, and relatively dense, that is, there
  exists $R>0$ such that every ball in $\R^d$ of radius $R$ contains at least one element of~$\Lambda$.
\end{definition}

\begin{definition}
  We say that two Delone sets $\Lambda,\Lambda'\subset\R^d$ are bounded distance equivalent, denoted
  by $\Lambda\sim_{bd}\Lambda'$, if there exist a constant $C>0$ and a bijection
  ${g}:\Lambda\to\Lambda'$ such that $|x-g(x)|<C$ for all $x\in\Lambda$.

  A Delone set $\Lambda\subset\R^d$ is bounded distance equivalent to a lattice (BDL), if there
  exist a $d$-dimensional lattice $L$ such that $\Lambda\sim_{bd}L$.
\end{definition}


Let us point out several simple properties of BDL sets:
\begin{itemize}
\item
  Let $x \in \mathbb{R} $. Then $\Lambda$ is BDL if and only if $\Lambda + x$ is BDL.
\item
  Let $\alpha \in \mathbb{R}, \alpha\neq 0 $.  Then $\Lambda$ is BDL if and only if $\alpha\Lambda$
  is BDL.
\item
  Let $F\subset \mathbb{R}$ be finite. Then $\Lambda$ is BDL if and only if $\Lambda \cup F$ is BDL.
\end{itemize}

In the one-dimensional case, the BDL property of a Delone set $\Lambda\subset\R$ is equivalent to
the fact that $\Lambda$ has bounded discrepancy.

\begin{lemma}\label{l:1dBDL}
  A Delone set $\Lambda\subset\R$ is bounded distance equivalent to the lattice $\xi\Z$, $\xi>0$, if
  and only if there exists $K>0$ such that for every $N>0$ we have
  \begin{equation}\label{eq:prevodBDL}
    \Big|\#\big(\Lambda\cap [0,N)\big) - \xi^{-1} N \Big| \leq K \ \text{ and }\
    \Big|\#\big(\Lambda\cap [-N,0)\big) - \xi^{-1} N \Big| \leq K.
  \end{equation}
\end{lemma}

\begin{proof}
  Implication $\Rightarrow$ is a consequence of the following claim:

  \textit{If $\Lambda\sim_{bd}\xi\Z$, $\xi>0$, i.e., there exists a bijection $g:\xi\Z\to\Lambda$
  and a constant $C$ such that $|g(\xi n) - \xi n|< C$ for every $n$, then for each interval
  $[\alpha,\beta)$ the difference between $\#\big(\Lambda\cap [\alpha,\beta)\big)$ and
  $\#\big(\xi\Z\cap [\alpha,\beta)\big)$ is bounded by $2C\xi^{-1}-2$.}

  For the proof of the claim, realize that if $a<b$ then
  \begin{equation}\label{eq:poctymriz}
    \floor{(b-a)\xi^{-1}} \leq \#\big(\xi\Z\cap [a,b)\big) \leq
    \left\lceil(b-a)\xi^{-1}\right\rceil.
  \end{equation}
  Denote $\#\big(\xi \Z\cap [\alpha-C,\beta+C)\big) = k$. Then $\Lambda$ has in $[\alpha,\beta)$ at
  most $k$ elements, i.e.,
  \begin{equation}\label{eq:poctyhorni}
    \#\big(\Lambda\cap [\alpha,\beta)\big) \leq \left\lceil(\beta-\alpha+2C)\xi^{-1}\right\rceil <
      (\beta-\alpha)\xi^{-1} + 2C\xi^{-1} +1.
  \end{equation}
  On the other hand, if $\#\big(\xi \Z\cap [\alpha+C,\beta-C)\big) = j$, then $\Lambda$ has in
  $[\alpha,\beta)$ at least $j$ elements, i.e.,
  \begin{equation}\label{eq:poctydolni}
    \#\big(\Lambda\cap [\alpha,\beta)\big) \geq \floor{(\beta-\alpha-2C)\xi^{-1}} >
      (\beta-\alpha)\xi^{-1} - 2C\xi^{-1} -1.
  \end{equation}
  Inequalities~\eqref{eq:poctyhorni},~\eqref{eq:poctydolni}, together with~\eqref{eq:poctymriz} give
  the bound
  \[
  \Big|\#\big(\Lambda\cap [\alpha,\beta)\big) - \#\big(\xi\Z\cap [\alpha,\beta)\big) \Big|
      \leq 2C\xi^{-1}+2.
  \]
  Implication $\Rightarrow$ of the lemma is obtained from the claim choosing $[\alpha,\beta)=[0,N)$,
  or $[\alpha,\beta)=[-N,0)$.

  In order to prove opposite implication, realize that a Delone set $\Lambda$ can be written as an
  increasing sequence $\Lambda=\{x_n:n\in\Z\}$ with $x_{-1}<0\leq x_0$.  Define a bijection
  $g:\Lambda\to\xi\Z$ by $g(x_n)=\xi n$.
  Obviously, $\#\big(\Lambda \cap [0,x_n)\big)=n$ for $n\geq 0$ and
  $\#\big(\Lambda \cap[x_n,0)\big)=n$, for $n<0$, and thus inequalities~\eqref{eq:prevodBDL} for
  $N=x_n$ imply that for each $x\in\Lambda$ we have $|x-g(x)|<K$.
\end{proof}

As a consequence of the proof, we can state that any 1-dimensional BDL set $\Lambda $ can be written
in the form $\Lambda = \{x_n:n\in\Z\}$, where $(x_n)_{n \in \mathbb{Z}}$ is a strictly increasing
sequence of real numbers such that $x_{-1}<0\leq x_0$ and for some constants $\eta,C>0$ we have
$|x_n-\eta n|<C$ for every $n\in\Z$.


\section{Geometric representations of infinite words}\label{sec:geomrep}


Our aim is now to clarify the relation of BDL property of a one-dimensional Delone set with finitely
many distances between its neighbors and the balancedness of the corresponding infinite word, or,
said differently, of an infinite word and its geometric representation.

For the definition of a geometric representation of an infinite word $\bu$ we use the Parikh vectors
of its prefixes. Denote for simplicity
\begin{equation}\label{eq:parikh}
  \Parikh[n]:=\left\{\begin{array}{@{}r@{}l@{\quad}l}
           \Parikh & (u_{[0,n)}) & \text{if } n\geq 0 \\[2mm]
           -\Parikh & (u_{[n,0)}) & \text{otherwise}.
  \end{array}\right.
\end{equation}

\begin{definition}\label{d:geomrep}
  Let $\bu=\cdots u_{-2}u_{-1}u_0u_1u_2\cdots$ be an infinite word over a finite alphabet
  $\A=\{1,\ldots,d\}$ and $\ell_1,\ldots,\ell_d$ be positive numbers.  Set
  \[
  x_n:=(\ell_1,\dots,\ell_d)\Parikh[n].
  \]
  The Delone set $\Lambda_{\bu}:=\{x_n : n\in\Z\}$ is called a geometric representation of
  $\bu$ defined by the lengths $\ell_1,\ldots,\ell_d$.  We say that the geometric
  representation is non-trivial, if the set of lengths has at least two elements.
\end{definition}

The above notion is illustrated in Figure~\ref{fig:geom_repr}.

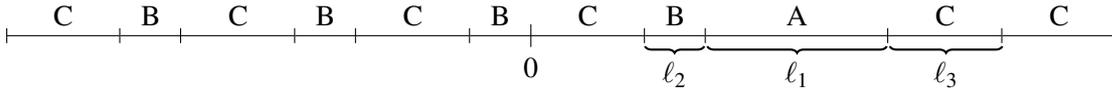
\begin{figure}[!htp]
    \centering
    \begin{tikzpicture}[x=2cm,y=1.5cm]

	\draw [|-|] (-1.15,0) -- node [anchor = south] {C} (-0.4,0);
          \draw [-|] (-0.4,0) -- node [anchor = south] {B} (0,0);
          \draw [-|] (0,0) -- node [anchor = south] {C} (0.75,0);
          \draw [-|] (0.75,0) -- node [anchor = south] {B} (1.15,0);
          \draw [-|] (1.15,0) -- node [anchor = south] {C} (1.9,0);
          \draw [-] (1.9,0) -- node [anchor = south] {B} (2.3,0);
          \draw [-|] (2.3,0) -- node [anchor = south] {C} (3.05,0);
          \draw [-|] (3.05,0) -- node [anchor = south] {B} (3.45,0);
          \draw [-|] (3.45,0) -- node [anchor = south] {A} (4.65,0);
          \draw [-|] (4.65,0) -- node [anchor = south] {C} (5.4,0);
          \draw [-|] (5.4,0) -- node [anchor = south] {C} (6.15,0);

          \draw (2.3,0.1) -- (2.3,-0.1) node [anchor = north] {0};


          \draw [thick,decorate,decoration={brace,mirror,raise=4pt}] (3.05,0) -- node [below=6pt] {$\ell_2$} (3.44,0);
	  \draw [thick,decorate,decoration={brace,mirror,raise=4pt}] (3.46,0) -- node [below=6pt] {$\ell_1$} (4.64,0);
	  \draw [thick,decorate,decoration={brace,mirror,raise=4pt}] (4.66,0) -- node [below=6pt] {$\ell_3$} (5.4,0);

        \end{tikzpicture}

        \caption{Geometric representation of (a part of) the infinite word $\cdots
          CBCBCB|CBACC\cdots$. The lengths $\ell_1,\ell_2,\ell_3$ correspond to letters $A,B,C$,
          respectively.}\label{fig:geom_repr}

\end{figure}


\begin{remark}
  Note that there are several cases of uninteresting geometric representations:
  \begin{enumerate}
  \item
    Every infinite word can be trivially represented by a lattice, if all lengths are chosen
    to be equal. In this case $\Lambda_{\bu}$ is a lattice itself.
  \item
    If there is a letter appearing in $\bu\in\A^{\Z}$ only finitely many times, say
    $a\in\A$ has $t$ occurrences in $\bu$, then the geometric representation of
    $\bu$ with $\ell_a=2$ and $\ell_b=1$ for all $b\in\A\setminus\{a\}$ is non-trivial.
    Moreover, for every $n\in\Z$ we have $|x_n-n|\leq t$, where
    $x_n=(\ell_1,\ldots,\ell_d)\Parikh[n]$. Therefore $\{x_n:n\in\Z\}$ is bounded distance
    equivalent to the lattice $\Z$.
  \item
    If $\bu$ is periodic with a period $w_1\cdots w_p\in\A^*$ then any choice of lengths
    $\ell_a$, $a\in\A$ such that $\ell_{w_1}+\ell_{w_2}+\cdots+\ell_{w_p}=p$ ensures that
    $|x_n-n|\leq p$ for every $n\in\Z$. Again, $\{x_n:n\in\Z\}$ is bounded distance
    equivalent to the lattice $\Z$.
  \end{enumerate}
  Therefore it is interesting to look for a non-trivial geometric BDL representation of a word
  $\bu$ only if $\bu$ is aperiodic such that every letter occurs in
  $\bu$ infinitely many times.
\end{remark}

It is a simple consequence of Proposition~\ref{p:frekvence} that every geometric representation of a
balanced infinite word is bounded distance equivalent to a lattice.

\begin{proposition}\label{p:anygeomBDL}
  Let $\bu$ be a bidirectional balanced infinite word over $\A$ and let $\varrho_a$,
  $a\in\A$ be frequencies of its letters. Then a geometric representation
  $\Lambda_{\bu}$ of $\bu$ defined by arbitrary choice of positive lengths $\ell_a$, $a\in\A$ is
  bounded distance equivalent to the lattice $\eta\Z$, where $\eta=\sum_{a\in\A}\ell_a\varrho_a$.
\end{proposition}

\begin{proof}
  Let $\bu$ be a $c$-balanced infinite word. Proposition~\ref{p:frekvence} implies for
  any letter $a\in\A$ and any $n\in\N$ that
  \begin{equation}\label{eq:skorohustoty}
    -c \leq |u_{[0,n)}|_a-\varrho_a n \leq c\quad\text{ and }\quad -c \leq |u_{[-n,0)}|_a-\varrho_a n \leq c.
  \end{equation}
  Multiplying by $\ell_a$ and summing over $a\in\A$, we obtain for the geometric representation of
  $\bu$ given by the sequence $\{x_n : n\in\Z\}$ that
  \[
  -\sum_{a\in\A} c\ell_a \leq x_n - n \sum_{a\in\A}\ell_a\varrho_a \leq \sum_{a\in\A} c\ell_a,
  \]
  see Definition~\ref{d:geomrep}. Lemma~\ref{l:1dBDL} then gives the result.
\end{proof}


\section{Fixed points of morphisms and balancedness}\label{sec:subst}


A special class of infinite words is formed by fixed points of morphisms. Let $\A$ be an alphabet
and $\varphi:\A^*\to\A^*$ a morphism.  We say that an infinite word $\bu$ over $\A$ is a fixed point
of $\varphi$ if $\varphi(\bu)=\bu$.  A morphism $\varphi$ is called substitution if it is
non-erasing, i.e., $\varphi(a) \neq \epsilon$ for each $a \in \mathcal{A}$, and there exist letters
$a,b \in \mathcal{A}$ such that $\varphi(a) =aw$ and $\varphi(b)= vb$ for some non-empty words $w,v
\in \mathcal{A}^*$.  Obviously, the bidirectional infinite sequence
$\cdots \varphi^3(v) \varphi^2(v) \varphi(v) b |a \varphi(w) \varphi^2(w) \varphi^3(w)\cdots$
is a fixed point of the substitution $\varphi$.

If $\bu$ is a fixed point of a morphism $\varphi$, the incidence matrix of $\varphi$ brings certain
information on the infinite word $\bu$. The incidence matrix $M_\varphi$ is obviously non-negative
with integer elements, i.e., $M_\varphi\in\N^{d\times d}$ where $d=\#\A$. It follows that its
spectral radius $r_\varphi$ is an eigenvalue of $M_\varphi$.  To the eigenvalue $r_\varphi$ there
exists an eigenvector with non-negative components~\cite{Fiedler}.  If the morphisms $\varphi$ is
non-erasing, then $r_\varphi \geq 1$.

If $M_\varphi$ is a primitive matrix (i.e., there exists a power of $M_\varphi$ which is positive),
the Perron-Frobenius theorem states that the eigenvalue $r_\varphi$ is simple and the corresponding
eigenvector is positive. It is (up to scaling) a unique non-negative eigenvector of $M_\varphi$. In
this case we say that $\varphi$ is a primitive morphism.

It is known~\cite{queffelec} that for a fixed point $\bu$ of a primitive morphism $\varphi$ the
frequencies of letters form an eigenvector
$\varrho_{\bu}=(\varrho_{a_1},\dots,\varrho_{a_d})^{\mathrm{T}}$ of the matrix $M_{\varphi}$ to its
spectral radius $r_{\varphi}$.
The transposed matrix $M_\varphi^{\mathrm{T}}$ has also $r_\varphi$ as an eigenvalue with a positive
eigenvector. If $\ell=(\ell_1,\dots,\ell_d)^{\mathrm{T}}$ is such an eigenvector, then it can be
used for a geometric representation of the fixed point $\bu$ of $\varphi$. Moreover, the geometric
representation $\Lambda_{\bu}$ obtained in this way is self-similar with the scaling factor
$r_\varphi$, i.e., $r_{\varphi}\Lambda_{\bu}\subset\Lambda_{\bu}$.

\begin{example}\label{ex:selfsimgeomrep}
  Consider the morphism $\varphi:\{A,B\}^*\to\{A,B\}^*$, given by $\varphi(A)=AAB$ and
  $\varphi(B)=AB$. The morphism $\varphi$ has the following fixed point
  \[
  \bu = \cdots AABABAABAB|AABAABABAABAABAB\cdots.
  \]
  The incidence matrix of the morphism $\varphi$ is
  ${M}_{\varphi}=\left(\begin{smallmatrix}2 & 1\\ 1 & 1 \end{smallmatrix}\right)$, its eigenvalues
  are $\frac12(3+\sqrt5)=\tau^2\sim 2.618$ and $\frac12(3-\sqrt5)=1/\tau^2\sim0.382$, where
  $\tau=\frac{1}{2}(1+\sqrt5)$ is the golden ratio.  The frequencies of letters $A,B$ in the
  infinite word $\bu$ are obtained as the components of the normalized positive eigenvector of
  $M_\varphi$, i.e., $\varrho_A=\tau^{-1}$, $\varrho_B=\tau^{-2}$. Since the matrix $M_\varphi$ is
  symmetric, the eigenvector $\varrho_{\bu}=(\tau^{-1},\tau^{-2})^{\mathrm{T}}$ can also be used to
  obtain a self-similar geometric representation of the infinite word $\bu$, see
  Figure~\ref{fig:geom_repr_(-tau)}.
\end{example}

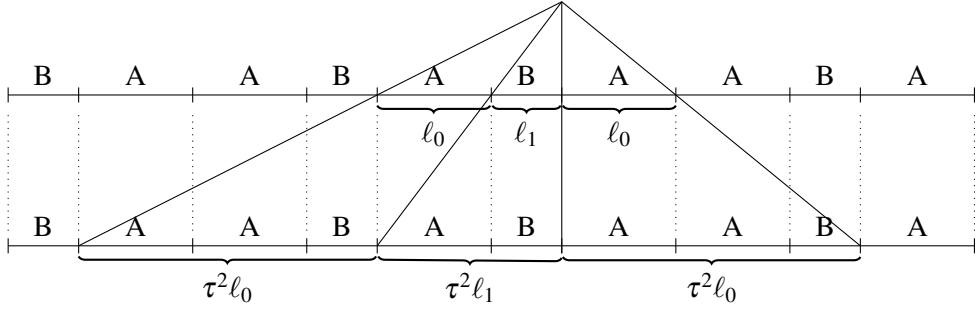
\begin{figure}[!htp]
  \centering

  \begin{tikzpicture}[x=1.5cm,y=2cm]

    \draw (-4.854,1)
      -- node [above] {B} (-4.236,1)
      -- node [above] {A} (-3.236,1)
      -- node [above] {A} (-2.236,1)
      -- node [above] {B} (-1.618,1)
      -- node [above] {A} (-0.618,1)
      -- node [above] {B} (0,1)
      -- node [above] {A} (1,1)
      -- node [above] {A} (2,1)
      -- node [above] {B} (2.618,1)
      -- node [above] {A} (3.618,1);

    \draw (-4.854,0)
      -- node [above] {B} (-4.236,0)
      -- node [above] {A} (-3.236,0)
      -- node [above] {A} (-2.236,0)
      -- node [above] {B} (-1.618,0)
      -- node [above] {A} (-0.618,0)
      -- node [above] {B} (0,0)
      -- node [above] {A} (1,0)
      -- node [above] {A} (2,0)
      -- node [above] {B} (2.618,0)
      -- node [above] {A} (3.618,0);

    \draw (-4.854,0.95)--(-4.854,1.05);
    \draw (-4.236,0.95)--(-4.236,1.05);
    \draw (-3.236,0.95)--(-3.236,1.05);
    \draw (-2.236,0.95)--(-2.236,1.05);
    \draw (-1.618,0.95)--(-1.618,1.05);
    \draw (-0.618,0.95)--(-0.618,1.05);
    \draw (0,0.95)--(0,1.05);
    \draw (1,0.95)--(1,1.05);
    \draw (2,0.95)--(2,1.05);
    \draw (2.618,0.95)--(2.618,1.05);
    \draw (3.618,0.95)--(3.618,1.05);
	
    \draw (-4.854,-0.05)--(-4.854,0.05);
    \draw (-4.236,-0.05)--(-4.236,0.05);
    \draw (-3.236,-0.05)--(-3.236,0.05);
    \draw (-2.236,-0.05)--(-2.236,0.05);
    \draw (-1.618,-0.05)--(-1.618,0.05);
    \draw (-0.618,-0.05)--(-0.618,0.05);
    \draw (0,-0.05)--(0,0.05);
    \draw (1,-0.05)--(1,0.05);
    \draw (2,-0.05)--(2,0.05);
    \draw (2.618,-0.05)--(2.618,0.05);
    \draw (3.618,-0.05)--(3.618,0.05);	
	
    \draw (0,-0.05) -- (0,1.618);
    \draw (0,1.618) -- (2.618,0);
    \draw (0,1.618) -- (-1.618,0);
    \draw (0,1.618) -- (-4.236,0);

    \draw [dotted] (-4.854,0.1) -- (-4.854,0.9);
    \draw [dotted] (-4.236,0.1) -- (-4.236,0.9);
    \draw [dotted] (-3.236,0.1) -- (-3.236,0.9);
    \draw [dotted] (-2.236,0.1) -- (-2.236,0.9);
    \draw [dotted] (-1.618,0.1) -- (-1.618,0.9);
    \draw [dotted] (-0.618,0.1) -- (-0.618,0.9);
    \draw [dotted] (1,0.1) -- (1,0.9);
    \draw [dotted] (2,0.1) -- (2,0.9);
    \draw [dotted] (2.618,0.1) -- (2.618,0.9);
    \draw [dotted] (3.618,0.1) -- (3.618,0.9);


    \draw [thick,decorate,decoration={brace,mirror,raise=4pt}]
      (-1.618,1) -- node [below=6pt] {$\ell_0$} (-0.628,1);
    \draw [thick,decorate,decoration={brace,mirror,raise=4pt}]
      (-0.608,1) -- node [below=6pt] {$\ell_1$} (-0.01,1);
    \draw [thick,decorate,decoration={brace,mirror,raise=4pt}]
      (0.01,1) -- node [below=6pt] {$\ell_0$} (1,1);
	
    \draw [thick,decorate,decoration={brace,mirror,raise=4pt}]
      (-4.236,0) -- node [below=6pt] {$\tau^2\ell_0$} (-1.628,0);
    \draw [thick,decorate,decoration={brace,mirror,raise=4pt}]
      (-1.608,0) -- node [below=6pt] {$\tau^2\ell_1$} (0,-0.01);
    \draw [thick,decorate,decoration={brace,mirror,raise=4.2pt}]
      (0.01,0) -- node [below=6pt] {$\tau^2\ell_0$} (2.618,0);	
	
  \end{tikzpicture}
	
  \caption{Self-similarity of the geometric representation of
		$\bu$.}\label{fig:geom_repr_(-tau)}
\end{figure}

In~\cite{AdamczBalance}, Adamczewski provides a description of the balance property of fixed points
of primitive morphisms. We include here only a simplified version of his result.

\begin{theorem}[\cite{AdamczBalance}]\label{t:Adam}
  Let $\bu$ be a fixed point of a primitive morphism $\varphi$ and $\sigma_{\varphi}$ be
  the spectrum of $M_{\varphi}$.
  \begin{itemize}
  \item
    If $|\lambda|<1$ for all $\lambda\in\sigma_{\varphi}\setminus\{r_{\varphi}\}$, then
    $\bu$ is balanced;
  \item
    if $|\lambda|>1$ for some $\lambda\in\sigma_{\varphi}\setminus\{r_{\varphi}\}$,
    then $\bu$ is not balanced.
  \end{itemize}
\end{theorem}

Note that Adamczewski describes balance property of fixed points even in case that the morphism has
an eigenvalue in modulus equal to one.


\section{The BDL  property of fixed points of substitutions}\label{sec:BDLsubst}


By Proposition~\ref{p:anygeomBDL}, any geometric representation of a balanced infinite word is BDL.
A natural question to ask is if any Delone set which has finitely many distances between neighboring
elements and is bounded distance equivalent to a lattice is coded by an infinite word which is
balanced.  We will show that this is not the case. In fact we provide a class of infinite words
which are not balanced, but have a non-trivial geometric representation which is BDL.  The class is
given by fixed points of substitutions.

First let us reformulate the BDL property of a geometric representation
$\Lambda_{\bu}=\{x_n:n\in\Z\}$ of an infinite word $\bu$ using vectors $\Parikh[n]$ given
by~(\ref{eq:parikh}).  Recall that $\{x_n:n\in\Z\}$ is BDL to a lattice if there exist $\eta, C>0$
such that $|x_n-\eta n|<C$ for all $n\in\Z$. Realizing that $n=(1,\dots,1)\Parikh[n]$, one can
reformulate the BDL property.

\begin{lemma}\label{l:vrstva}
  Let $\bu$ be a bidirectional infinite word over an alphabet $\A=\{1,\ldots,d\}$.
  Then the geometric
  representation of $\bu$ with positive lengths $\ell_1,\ldots,\ell_d$ is bounded distance
  equivalent to the lattice $\eta\Z$ if and only if for some constant $C$
  \[
  \Big|(\ell_1-\eta,\dots,\ell_d-\eta)\Parikh[n]\Big|<C\qquad\text{ for every } n\in\Z.
  \]
\end{lemma}

\begin{lemma}\label{l:matice}
  Let $\bu$ be a fixed point of a substitution $\varphi$ over a $d$-letter alphabet $\A$ and suppose
  that the incidence matrix of $\varphi$ has at least one eigenvalue in modulus less than 1. Then
  there exist a vector $f\in\R^d$ linearly independent of $(1,1,\dots,1)^{\mathrm{T}}$ and a
  constant $K>0$ such that $|f^{\mathrm{T}}\Parikh[n]|\leq K$ for every $n\in\Z$.
\end{lemma}

\begin{proof}
  Let $M\in\N^{d\times d}$ be the incidence matrix of $\varphi$. Since $\varphi$ is non-erasing
  there is a non-zero eigenvalue of $M$. Moreover, as $M$ is an integer matrix, there exists
  an eigenvalue of $M$ which is in modulus greater than or equal to $1$.

  Write $\R^d$ as a direct sum $\R^d=V_1\oplus V_0$ where $V_1$, $V_0$ are invariant subspaces of
  $M$ such that $M$ restricted to $V_1$ has all eigenvalues in modulus $\geq 1$ and $M_\varphi$
  restricted to $V_0$ has all eigenvalues in modulus $< 1$.  Necessarily $V_0\neq\{0\}$, $V_1\neq\{0\}$.

  Denote $\beta:=\max\big\{|\lambda|: \lambda\text{ is an eigenvalue of } M,\ |\lambda|<1\big\}$.
  For each $\Parikh[n]$, denote $Y_n\in V_1$, $Z_n\in V_0$ such that $\Parikh[n]=Y_n+Z_n$. Choose
  $\varepsilon>0$ such that $\beta+\varepsilon<1$. By Theorem 3 from~\cite{isaacson} there exists a
  norm $\|\cdot\|_0$ on the space $V_0$ such that for every $Z\in V_0$, we have $\|M Z\|_0\leq
  (\beta+\varepsilon)\|Z\|_0$. We show that there exists a constant $c>0$ such that
  \begin{equation}\label{eq:vektory1}
    \|Z_n\|_0\leq c \qquad \text{ for every } n\in\Z.
  \end{equation}
  In order to obtain the constant $c$ denote
  \begin{align*}
    {\mathcal P} & :=\{p\in\A^* : p \text{ is a proper prefix or a proper suffix of }\varphi(a),\ a\in\A\},\\
    \gamma &:= \max\big\{ \|Z\|_0 : Z\in V_0, \text{ such that } \exists\, Y\in V_1,\exists p\in{\mathcal P},
    \Parikh(p)=Y+Z \big\}.
  \end{align*}
  Find $c>0$ so that
  \begin{equation}\label{eq:vektory2}
    (\beta+\varepsilon)c+\gamma \leq c.
  \end{equation}
  This is possible, by the choice of $\varepsilon$. We will show~\eqref{eq:vektory1} by induction on
  $n$. For $n=0$, trivially $\Parikh[0]=\Parikh(\epsilon)=0$, and so $Z_0=0$. Let $n> 0$.
  Then there exist
  $m<n$ and $p\in{\mathcal P}$ such that $u_{[0,n)}=\varphi(u_{[0,m)})p$. Since
  $\Parikh\big(\varphi(w)\big)=M\Parikh(w)$, for every $w\in\A^*$, we derive
  $\Parikh[n]=M\Parikh[m]+\Parikh(p)$. If $Y\in V_1$, $Z\in V_0$ are such that $\Parikh(p)=Y+Z$, then
  $\Parikh[n]=M Y_m + M Z_m + Y + Z$. The decomposition of a vector to components in the invariant
  subspaces $V_0$, $V_1$ is unique and thus $Z_n=MZ_m + Z$. This implies
  \[
  \|Z_n\|_0 \leq \|MZ_m\|_0 + \|Z\|_0 \leq (\beta+\varepsilon)\|Z_m\|_0 + \gamma.
  \]
  The induction hypothesis and the choice of $c$ by~\eqref{eq:vektory2} gives
  \[
  \|Z_n\|_0 \leq (\beta+\varepsilon)c + \gamma\leq c.
  \]
  The proof of~\eqref{eq:vektory1} for $n<0$ is analogous. Since on a finitely dimensional space
  all norms are equivalent, we have thus shown that the set $\{Z_n:n\in\Z\}$ is bounded in any norm.

  Consider the subspace of vectors $f\in\R^d$ such that
  \begin{equation}\label{eq:vektory3}
    f^{\mathrm{T}} Y=0 \qquad \text{ for every } Y\in V_1.
  \end{equation}
  If $\dim V_1<d-1$, the subspace is of dimension $\geq 2$. Therefore one can find a vector
  $f\in\R^d$ with at least two distinct components.

  If $\dim V_1=d-1$, then $\dim V_0=1$ and there is a unique eigenvalue $\lambda$ of the matrix $M$,
  such that $|\lambda|=\beta<1$. Necessarily, $\lambda$ is a simple real eigenvalue of $M$ and any
  non-zero $f$ satisfying~\eqref{eq:vektory3} is an eigenvector to the eigenvalue $\lambda$. Choose
  one such $f$. Since $M^{\mathrm{T}}$ is a non-negative matrix, its spectral radius is its
  eigenvalue and there exists a non-negative eigenvector $v$ corresponding to this
  eigenvalue~\cite{Fiedler}. As $\lambda$ is in modulus less than $1$, $\lambda$ is not equal to the
  spectral radius of $M$, and thus the vector $f$ is orthogonal to the non-negative vector
  $v$. Therefore $f$ cannot have all components equal. Since by~\eqref{eq:vektory1} the set
  $\{Z_n:n\in\Z\}$ is bounded in $\R^d$, necessarily $\{f^{\mathrm{T}} Z_n : n\in\Z\}$ is a bounded
  set in $\R$, i.e., there exists a constant $K>0$ such that
  \[
  \big|f^{\mathrm{T}} \Parikh[n]\big| = |f^{\mathrm{T}} (Y_n+Z_n)| = |f^{\mathrm{T}} Z_n|\leq K,
  \]
  which we had to show.
\end{proof}

\begin{proposition}\label{prop:BDL_repr_of_fixed_point}
  Let $\varphi$ be a substitution over an alphabet $\A$ and suppose that its incidence matrix has at
  least one eigenvalue in modulus less than 1.  Let $\bu$ be a bidirectional fixed point of
  $\varphi$. Then there exists a non-trivial geometric representation of $\bu$ which is bounded
  distance equivalent to a lattice.
\end{proposition}

\begin{proof}
  Denote $d:=\#\A$. By Lemma~\ref{l:matice}, there exists a vector
  $f=(f_1,\dots,f_d)^{\mathrm{T}}\in\R^d$ with at least two distinct components such that
  $f^{\mathrm{T}}\Parikh[n]$ is bounded, where $\Parikh[n]$ is defined by~\eqref{eq:parikh}. Take
  any $\eta>0$ such that $f+\eta(1,1,\dots,1)^{\mathrm{T}}=(f_1+\eta,\dots,f_d+\eta)^{\mathrm{T}}$
  is a positive vector.  By Lemma~\ref{l:vrstva}, the geometric representation of $\bu$ with lengths
  $\ell_i=f_i+\eta$, $i=1,\dots,d$, is bounded distance equivalent to the lattice $\eta\Z$. Since
  the vector $f$ has at least two distinct components, the representation is non-trivial.
\end{proof}

The following example illustrates that also non-balanced infinite words may have non-trivial
geometric representation which is BDL.

\begin{example}
  Let us consider morphism $\varphi: \{A,B,C\}^*\to\{A,B,C\}^*$ given by $A\mapsto C$, $B\mapsto
  ACCCC$, and $C\mapsto CB$. Note that $\varphi$ is not a substitution but $\varphi^2$ is. Denote by
  $\bu$ the infinite word obtained by iteration of $\varphi^2$ starting from the pair $B|C$. Since
  ${M}_{\varphi}$ has an eigenvalue $\lambda\sim-0.274$ (and thus ${M}_{\varphi^2}={M}_{\varphi}^2$
  has an eigenvalue $\lambda^2$ such that $|\lambda^2|<1$) we can use
  Proposition~\ref{prop:BDL_repr_of_fixed_point} to find a geometric representation of $\bu$ which
  is BDL. The lengths of this representation are $\ell_i=f_i+\eta$, where $f_i$ are elements of the
  vector $f$ from the proof of Lemma~\ref{l:matice}.  Since ${M}_{\varphi^2}$ has only one
  eigenvalue which is in modulus smaller than 1 (namely $\lambda^2$), the vector $f$ is an
  eigenvector to $\lambda^2$. Therefore
  \[
  f = \begin{pmatrix} 1 \\ \lambda^2 - \lambda \\ \lambda \end{pmatrix} \sim
  \begin{pmatrix} 1 \\ 0.349 \\ -0.274 \end{pmatrix},
  \]
  and we see that taking any $\eta>|\lambda|$ we get the sought geometric representation.
  Such representation is non-trivial, and, moreover, the lengths associated to distinct letters are distinct.
  As the spectrum of the matrix $M_\varphi$ contains two eigenvalues in modulus $>1$, the fixed point
  $\bu$ is not balanced, see Theorem~\ref{t:Adam}.
\end{example}

The following example illustrates that the assumption of the existence of an eigenvalue in modulus smaller
than one cannot be omitted.

\begin{example}
  Let $\varphi:\{A,B\}^*\to\{A,B\}^*$ be the substitution given by $A\mapsto ABBA$ and $B\mapsto AA$.
  It incidence matrix $M_{\varphi}=\left(\begin{smallmatrix}2 & 2 \\ 2 & 0\end{smallmatrix}\right)$ has
  eigenvalues $\lambda_1\sim3.23607$ and $\lambda_2\sim-1.23607$. Let $\bu$ be a fixed point of $\varphi$
  generated starting from $A|A$.

  We consider a subsequence of $(\Parikh[n])_{n\geq 0}$, namely
  $\big(\Parikh(\varphi^n(A))\big)_{n\geq 0}$. Using the incidence matrix of $\varphi$ we can write
  \[
  \Parikh\big(\varphi^n(A)\big) = M_{\varphi}^n\Parikh(A) =
	2^n\begin{pmatrix}1 & 1 \\ 1 & 0 \end{pmatrix}^{\!n}\!\!
	\begin{pmatrix}1\\0\end{pmatrix} = 2^n \begin{pmatrix} \mathrm{F}_{n+1} \\ \mathrm{F}_n \end{pmatrix},
  \]
  where $(F_n)_{n\geq 0}$ is the Fibonacci sequence.

  Should the statement of Lemma~\ref{l:matice} hold for $\varphi$ there would be a vector
  $f=\left(\begin{smallmatrix}f_1\\f_2\end{smallmatrix}\right)$ with $f_1\neq f_2$ and a constant
  $K>0$ such that
  \[
  \big|2^n(f_1\mathrm{F}_{n+1}+f_2\mathrm{F}_n)\big| =
  \big|f^{\mathrm{T}}\Parikh\big(\varphi^n(A)\big)\big| \leq K
  \qquad \text{for all $n\in\N$}.
  \]
  Using the Binet formula for the $n$-th term of the Fibonacci sequence
  \[
  \mathrm{F}_n = \frac{1}{\sqrt{5}}\big(\tau^{n+1}-(\tau')^{n+1}\big),
  \]
  where $\tau$ is the golden mean and $\tau'$ is its conjugate, we have
  \begin{align*}
  \big|2^n(f_1\mathrm{F}_{n+1}+f_2\mathrm{F}_n)\big| &=
	\left|2^n\left(\frac{f_1}{\sqrt{5}}\big(\tau^{n+2}-(\tau')^{n+2}\big)\right)+
	\left(\frac{f_2}{\sqrt{5}}\big(\tau^{n+1}-(\tau')^{n+1}\big)\right)\right| \\
	&= \frac{1}{2\sqrt{5}}\bigg|(2\tau)^{n+1}\bigg(f_1\tau+f_2-
	\underbrace{(f_1\tau'+f_2)\Big(\frac{\tau'}{\tau}\Big)^{n+1}}_{(\star)}\bigg)\bigg|.
  \end{align*}
  Since the term $(\star)$ goes to 0 as $n$ goes to infinity, the only case when
  $\big|f^{\mathrm{T}}\Parikh\big(\varphi^n(A)\big)\big|$ can be bounded is the case where
  $f_2=-f_1\tau$. In such a case
  \begin{align*}
  \big|2^n(f_1\mathrm{F}_{n+1}+f_2\mathrm{F}_n)\big| &=
  \frac{1}{2\sqrt{5}}\bigg|(2\tau)^{n+1}\bigg(\underbrace{f_1\tau-f_1\tau}_{=\:0}
	-(f_1\tau'-f_1\tau)\Big(\frac{\tau'}{\tau}\Big)^{n+1}\bigg)\bigg| \\
	&= \frac{1}{2\sqrt{5}}\bigg|(2\tau)^{n+1}\bigg(f_1(\underbrace{\tau-\tau'}_{=\:\sqrt{5}})
	\Big(\frac{\tau'}{\tau}\Big)^{n+1}\bigg)\bigg| = \frac{1}{2}\Big|f_1(2\tau')^{n+1}\Big|.
  \end{align*}
  The last expression tends to infinity (recall that $2\tau'\sim-1.236$ and $f_1$ is non-zero,
  otherwise $f_1=f_2$), thus $\big|f^{\mathrm{T}}\Parikh\big(\varphi^n(A)\big)\big|$ is not bounded.
\end{example}




\section{The BDL property of cut-and-Project sequences}\label{sec:BDLcap}


The object in focus, spectrum of quadratic Pisot units, was identified in~\cite{MaPaPe} with the
cut-and-project sequences, one-dimensional Delone sets obtained by projection of a section of the
two-dimensional integer lattice $\Z^2$ to an irrationally oriented straight line. The construction
is illustrated in Figure~\ref{fig:capscheme}.

\begin{figure}[!htp]
  \centering
  \begin{tikzpicture}[x=1cm,y=1cm]
    \draw [step=1.25cm, gray] (-0.25,-0.25) grid (10.25,6.5);
    \draw [name path = hlavni, thick] (-1.5,-0.25) -- (23.5:12)
      node [below right, xshift = -5] {$y=\varepsilon x$};
    \draw [yshift = 55] (-1.5,-0.25) -- (23.5:12);
    \draw (1.45,-0.25) -- ($(1.45,-0.25)+(113.5:7)$)
      node [below right, xshift = 7, fill=white] {$y=\eta x$};

    \foreach \bod [count=\i] in
	     {{0,1.25},{1.25,1.25},{1.25,2.5},{2.5,2.5},{3.75,2.5},{5,2.5},{5,3.75},
	     {6.25,3.75},{7.5,3.75},{7.5,5},{8.75,5},{10,5},{3.75,3.75},{10,6.25}} {
      \fill [black] (\bod) circle (1.5pt);
      \path [ name path = perp-\i ] (\bod) -- ($(\bod)+(-66.5:1.9)$);
      \draw [name intersections={of=hlavni and perp-\i, by=inter-\i}] (\bod) -- (inter-\i);
      \draw [very thick] (inter-\i) -- ($(inter-\i)+(-66.5:0.1)$);
    }
  \end{tikzpicture}

  \caption{The cut-and-project scheme.}\label{fig:capscheme}
\end{figure}
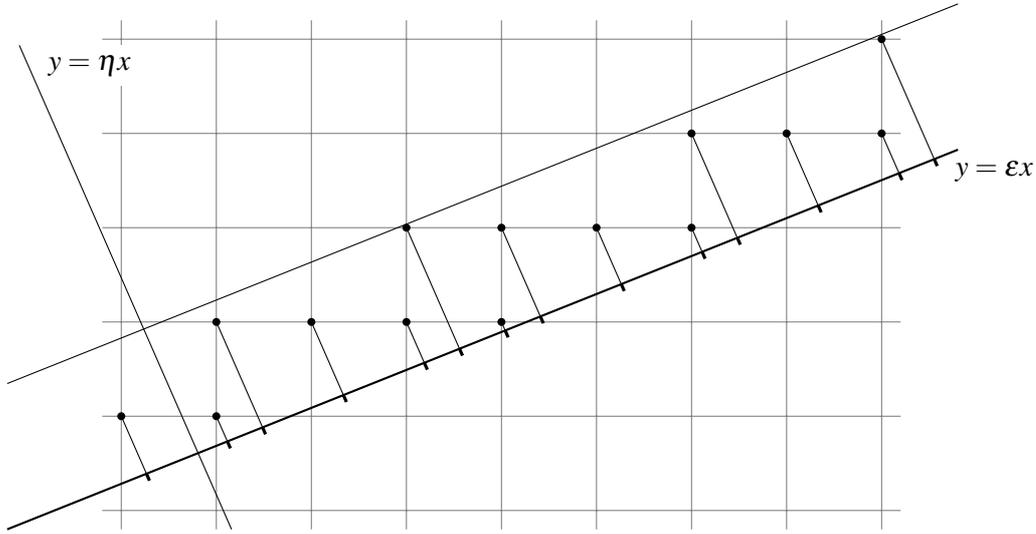

It can be shown that such a construction leads to the following definition.

\begin{definition}
  Let $\varepsilon,\eta$ be irrational, $\varepsilon\neq\eta$, and let $\Omega$ be a bounded
  interval. Let $\star:\Z+\Z\eta\to\Z+\Z\varepsilon$ be the isomorphism between additive groups
  $\Z+\Z\eta$ and $\Z+\Z\varepsilon$ given by $(a+b\eta)^{\star}=a+b\varepsilon$. The set
  \[
  \Sigma_{\varepsilon,\eta}(\Omega) = \big\{x\in\Z+\Z\eta:x^{\star}\in\Omega\big\}
  \]
  is called a cut-and-project set with acceptance window $\Omega$.
\end{definition}

The symmetries of the integer lattice $\Z^2$ give rise to certain symmetries of the
cut-and-project sets.
As $\Z^2$ is preserved under the multiplication by a unimodular integer matrix
$\left(\!\begin{smallmatrix}A & B \\[1pt] C & D \end{smallmatrix}\!\right)$, we have
\begin{equation}\label{eq:transformaceCAP}
  \Sigma_{\varepsilon,\eta}(\Omega) = (A+C\eta)
  \Sigma_{\tfrac{B+D\varepsilon}{A+C\varepsilon},
  \tfrac{B+D\eta}{A+C\eta}}\left(\frac1{A+C\varepsilon}\Omega\right)\,.
\end{equation}
Because of the translation symmetry of the integer lattice $\Z^2$, we have
\begin{equation}\label{eq:translaceCAP}
  \Sigma_{\varepsilon,\eta}(\Omega) + x =  \Sigma_{\varepsilon,\eta}(\Omega+x^\star)
  \qquad\text{ for every \ } x\in\Z+\Z\eta.
\end{equation}
As $\eta$ and $\varepsilon$ are irrational, the sets $\Z+\Z\eta$ and $\Z+\Z\varepsilon$ are dense in
$\mathbb{R}$.

With this is mind, we see that without loss of generality, we can assume that the acceptance
interval $\Omega$ is of the form $\Omega=[c,d)$ for some $0\leq c< d <1$. Note that the
presence/absence of the boundary points in the acceptance interval may cause presence/absence of
at most two points in the cut-and-project sets. The semi-open interval is chosen for convenience.

The lattice bounded distance property of cut-and-project sets has been studied by different
authors. In one dimension, results can be found for example in~\cite{Fret1}.  We rephrase the result
in this section in our formalism and put it into context of infinite words.

The property of one-dimensional cut-and-project sets to be or not bounded distance equivalent to a
lattice depends on the length of the acceptance interval. The following is a simple consequence of a
result of Kesten~\cite{Kesten}.

\begin{proposition}\label{coro:Kesten}
  Let $\varepsilon,\eta$ be irrational such that $\varepsilon\in(0,1)$, $\varepsilon\neq\eta$, and
  let $0\leq c<d\leq 1$. Then the set
  \[
  \Sigma_{\varepsilon,\eta}\big([c,d)\big) = \big\{a+b\eta : a,b\in\Z,\ c\leq a+b\varepsilon < d \big\}
  \]
  is BDL if and only if $d-c\in\Z+\Z\varepsilon$. If it is the case, then it is BDL to the lattice
  $\frac{|\eta-\varepsilon|}{d-c}\Z$.
\end{proposition}

Before stating the proof, we cite the following characterization proved in~\cite{Kesten}.

\begin{theorem}[\cite{Kesten}]\label{t:Kesten}
  Let $\varepsilon\in(0,1)$, $0\leq c<d\leq 1$. Then
  \[
  \#\big\{1\leq k\leq N : k\in\Z,\ c\leq \{k\varepsilon\}<d \big\} - N(d-c)
  \]
  is bounded in $N$ if and only if $d-c\in\Z+\Z\varepsilon$.
\end{theorem}


\begin{proof}[Proof of Proposition~\ref{coro:Kesten}]
  Consider first the set
  \[
  \Sigma_{\varepsilon,\eta}\big([0,1)\big)=\{ a+b\eta : a,b\in\Z,\ a+b\varepsilon \in[0,1) \}.
  \]
  Since $a+b\varepsilon \in[0,1)$ if and only if $b\varepsilon -1 < -a \leq b\varepsilon$, or
  equivalently $-a = \lfloor b\varepsilon\rfloor$, we can write
  \begin{equation}\label{eq:jednotkoveokno}
    \Sigma_{\varepsilon,\eta}\big([0,1)\big)=\{ b\eta -\floor{b\varepsilon}  : b\in\Z \}.
  \end{equation}
  Therefore
  \begin{align*}
    \Sigma_{\varepsilon,\eta}\big([c,d)\big) & =
      \{ n\eta - \floor{n\varepsilon} : n\in\Z,\ n\varepsilon - \floor{n\varepsilon} \in[c,d) \} =\\
    & =\{ n(\eta-\varepsilon) + \{n\varepsilon\} : n\in\Z,\ c\leq \{n\varepsilon\} <d \}.
  \end{align*}
  In order to decide about the BDL property of the set $\Sigma_{\varepsilon,\eta}\big([c,d)\big)$,
  we will use Lemma~\ref{l:1dBDL}.  Let us to determine the number of elements of
  $\Sigma_{\varepsilon,\eta}\big([c,d)\big)$ in the intervals $[0,N)$ and $[-N,0)$.  Assume first
  that $\eta>\varepsilon$. Then we have
  \[
  0 \leq n(\eta-\varepsilon) + \{n\varepsilon\} < N
  \]
  if and only if
  \[
  -\{n\varepsilon\}(\eta-\varepsilon)^{-1} \leq n <
    N(\eta-\varepsilon)^{-1} -\{n\varepsilon\})(\eta-\varepsilon)^{-1}.
  \]
  The number of such integers $n$ is equal to $\floor{N(\eta -\varepsilon)^{-1}}$ or
  $\left\lceil N(\eta -\varepsilon)^{-1}\right\rceil$.  The number of points in
  $\Sigma_{\varepsilon,\eta}\big([c,d)\big)$ in the interval $[0,N)$ is therefore equal to
  \[
  \#\big\{1\leq n\leq N(\eta -\varepsilon)^{-1} : n\in\Z,\ c\leq \{n\varepsilon\}<d \big\}
  \]
  up to a possible difference of a constant independent of $N$.
  Consequently, by Theorem~\ref{t:Kesten},  the sequence
  \[
  \#\big\{1\leq n\leq N(\eta -\varepsilon)^{-1} : n\in\Z,\ c\leq
    \{n\varepsilon\}<d \big\} - (d-c)(\eta -\varepsilon)^{-1}N
  \]
  is bounded in $N$ if and only if $d-c\in\Z[\varepsilon]$.  Similarly, one proceeds for the
  interval $[-N,0)$.  Lemma~\ref{l:1dBDL} implies that
  \[
  \#\Big(\Sigma_{\varepsilon,\eta}\big([c,d)\big) \cap [0,N)\Big) - (d-c)N(\eta -\varepsilon)^{-1}
  \]
  is bounded in $N$ if and only if $d-c\in\Z[\varepsilon]$. If it is the case, then
  $\Sigma_{\varepsilon,\eta}\big([c,d)\big)$ is bounded distance to the lattice $\xi \mathbb{Z}$,
  where $\xi = (d-c)^{-1}(\eta -\varepsilon)$.

  If $\eta<\varepsilon$, the proof is analogous.
\end{proof}

For any choice of parameters $\varepsilon,\eta,\Omega$, the distances between consecutive points in
the set $\Sigma_{\varepsilon,\eta}(\Omega)$ take only two or three values,
see~\cite{GuMaPeBordeaux}.  Then, one can define an infinite word $\bu_{\varepsilon,\eta}(\Omega)$
over a binary or a ternary alphabet which codes the ordering of the distances in the cut-and-project
sequence.  Given parameters $\varepsilon,\eta$, the values of the distances do not depend on the
position of the acceptance interval $\Omega$ but only on its length $|\Omega|$. Only for a discrete
set of values of $|\Omega|$, the cut-and-project sequence $\Sigma_{\varepsilon,\eta}(\Omega)$ gives
a binary infinite word $\bu_{\varepsilon,\eta}(\Omega)$. In that case, the infinite word is a
sturmian word, i.e., an aperiodic binary $1$-balanced word. On the other hand, any sturmian word can
be obtained as $\bu_{\varepsilon,\eta}(\Omega)$, which can be seen for example through the
equivalent definition of sturmian words by mechanical words. For the equivalence between these
definitions, see~\cite{lothaire,lunpl}.

\begin{example}\label{ExampleContinue}
  Take $\varepsilon=\frac12(3-\sqrt5)=\tau^{-2}\sim0.312$, $\eta=\frac12(3+\sqrt5)=\tau^2\sim2.618$
  and $\Omega=[0,1)$, then we obtain from~\eqref{eq:jednotkoveokno} that $\Sigma_{\varepsilon,\eta}(\Omega)$ has
  distances between consecutive points in the form
  \[
  (b+1)\tau^2 +\floor{(b+1)/\tau^2} - b\tau^2 - \floor{b/\tau^2} \in\{\tau^2,\tau^2-1=\tau\}.
  \]
  First few elements of $\Sigma_{\varepsilon,\eta}(\Omega)$ around zero are
  \[
  \dots, - \tau^2-\tau, - \tau, 0 , \tau^2, 2\tau^2, 2\tau^2+\tau, \dots.
  \]
  In Example~\ref{ex:examplecontinue2} we will show that $\Sigma_{\varepsilon,\eta}(\Omega)$ is a
  geometric representation of the infinite word from Example~\ref{ex:selfsimgeomrep}.
\end{example}

For a generic length of the interval $\Omega$, the infinite word $\bu_{\varepsilon,\eta}(\Omega)$
coincides with a coding of exchange of three intervals with permutation $(321)$,
see~\cite{GuMaPeBordeaux}.

\begin{corollary}\label{coro:ekvivalenceCAP}
  Let $\varepsilon,\eta$ be irrational, $\varepsilon\neq\eta$, and let\/ $\Omega$ be a bounded
  interval.  The following statements are equivalent:
  \begin{enumerate}[(a)]
  \item
    The cut-and-project sequence
    $\Sigma_{\varepsilon,\eta}(\Omega)$ is bounded distance to a lattice;
  \item
    the infinite word $\bu_{\varepsilon,\eta}(\Omega)$ is balanced;
  \item
    $|\Omega|\in\Z+\Z\varepsilon$.
  \end{enumerate}
  Moreover, if it is the case, then $\Sigma_{\varepsilon,\eta}(\Omega)$ is bounded distance to the
  lattice $\frac{|\eta-\varepsilon|}{|\Omega|}\Z$.
\end{corollary}

\begin{proof}
  First realize that without loss of generality we can assume $\varepsilon\in(0,1)$ and
  $\Omega\subset [0,1)$. Otherwise, applying~\eqref{eq:transformaceCAP} and \eqref{eq:translaceCAP}
  to the cut-and-project set $\Sigma_{\varepsilon,\eta}(\Omega)$ we find a cut and project set
  $\Sigma_{\tilde{\varepsilon},\tilde{\eta}}(\widetilde{\Omega})$ which is an affine image of
  $\Sigma_{\varepsilon,\eta}(\Omega)$ and its parameters already satisfy
  $\tilde{\varepsilon}\in (0,1)$ and $ \widetilde{\Omega} \subset [0,1)$.  Let us stress that the
  transformation~\eqref{eq:transformaceCAP} satisfies: $|\Omega|\in\Z+\Z\varepsilon$ if and
  only if $|\widetilde{\Omega}|\in\Z+\Z\tilde{\varepsilon}$.

  Equivalence (a) $\Leftrightarrow$ (c) has been established as Proposition~\ref{coro:Kesten}.

  For the proof of (a) $\Rightarrow$ (b) assume that $\Sigma_{\varepsilon,\eta}(\Omega)$ is bounded distance
  equivalent to a lattice. By Proposition~\ref{coro:Kesten} this means that
  $|\Omega|\in\Z+\Z\varepsilon$. Result of Cassaigne~\cite{Cassaigne} implies that the
  infinite word $\bu_{\varepsilon,\eta}(\Omega)$ is
  a morphic image of a sturmian word. Since sturmian words
  are $1$-balanced, Proposition~\ref{p:morphicimagebalance} gives that
  $\bu_{\varepsilon,\eta}(\Omega)$ itself is balanced.

  In order to show (b) $\Rightarrow$ (a) we realize that $\Sigma_{\varepsilon,\eta}(\Omega)$ is a
  geometric representation of $\bu_{\varepsilon,\eta}(\Omega)$. It suffices now to use
  Proposition~\ref{p:anygeomBDL}.

  In order to derive to which lattice the cut-and-project sequence is BDL, we again use
  Proposition~\ref{coro:Kesten} which states that if $\varepsilon\in(0,1)$ and $|\Omega|<1$, then
  $\Sigma_{\varepsilon,\eta}(\Omega)$ is bounded distance equivalent to the lattice
  $\frac{|\eta-\varepsilon|}{|\Omega|}\Z$. If we need to use
  transformation~\eqref{eq:transformaceCAP}, we derive that
  $\Sigma_{\tilde{\varepsilon},\tilde{\eta}}(\widetilde{\Omega})$ is BDL to $\xi\Z$ where
  \[
  \xi=\left|\frac{\tilde{\eta}-\tilde{\varepsilon}}{|\widetilde{\Omega}|}\right|
  =\left|\frac{\frac{B+D\eta}{A+C\eta}-\frac{B+D\varepsilon}{C+D\varepsilon}}
    {\frac{1}{A+C\varepsilon}|\Omega|}\right|
  =\frac{|\eta-\varepsilon|}{|A+C\eta|\cdot|\Omega|}.
  \]
  Therefore $\Sigma_{\varepsilon,\eta}(\Omega) =
  (A+C\eta)\Sigma_{\tilde{\varepsilon},\tilde{\eta}}(\widetilde{\Omega})$ is BDL to the lattice
  $\frac{|\eta-\varepsilon|}{|\Omega|}\Z$.
\end{proof}

\section{The spectrum of quadratic  units  and  the BDL property }\label{sec:spectrum}



The spectra of algebraic numbers, as defined in~\eqref{eq:defspectrum}, are studied since their
introduction by Erd\H os et al.~\cite{Erdos} who considered the case $\alpha\in(1,2)$ and
$\D=\{0,1,\dots,M\}$, $M\in\N$. The interest of the spectra~\eqref{eq:defspectrum} stems in the
direct connection of their geometric properties to the arithmetic properties of the corresponding
numeration system. For example, relative density of the spectrum is equivalent to completeness of
the numeration system~\cite{Vavroch}, uniform discreteness has impact on the
effectivity of algorithms~\cite{FrPe}. In case of real base $\alpha$ and real alphabet, many of the
properties are known. If the base is $\alpha=\pm\beta$ where $\beta$ is a Pisot number and the
alphabet of digits is formed by sufficiently many consecutive integers, the spectrum is a Delone set
of finite local complexity~\cite{bugeaud} and the sequence of distances between neighboring
elements in the spectrum is coded by an infinite word which is substitutive~\cite{FengWen}.

Note that if the base $\alpha$ is positive and the alphabet of digits non-negative, then the
spectrum lies in $[0,+\infty)$, and, obviously, cannot be bounded distance equivalent to any
  lattice. It leads us to consider the generalized spectrum with either negative base, or the
  alphabet of digits containing $\pm1$.  It is also reasonable to require that the alphabet of
  digits contains $\#\D>|\alpha|$ elements, since otherwise the numeration system with base $\alpha$
  and digits $\D$ is not complete, and consequently, the spectrum is not relatively dense.

In this paper, we are interested in the spectrum as considered in~\cite{MaPaPe}: Let
$\D=\{m,\dots,M\}$, $m \leq 0 < M$, and let $\beta=|\alpha|$ be a quadratic Pisot unit, i.e.,
$\beta>1$ an algebraic number whose minimal polynomial is of the form
\[
  x^2 - px - 1,\ p \geq 1
  \qquad\text{or}\qquad
  x^2 -  px + 1,\ p \geq 3.
\]
We have $\lfloor\beta\rfloor=p$ and $\lceil\beta\rceil=p$ in the first and second case,
respectively.

In~\cite{MaPaPe} it was shown that, in this case, the spectrum $X^{\D}(\alpha)$ can be identified
with a cut-and-project set $\Sigma_{\varepsilon,\eta}(\Omega)$ with parameters $\eta=\beta$ and
$\varepsilon=\beta'$, where $\beta'$ is the conjugate of $\beta$.


\begin{proposition}[\cite{MaPaPe}]\label{t:MaPaPe}
  Let $\beta>1$ be a quadratic unit with conjugate $\beta'$. Let $\D\ni0$ be an alphabet of
  consecutive integers satisfying $\#\D>\beta$. If
  \[
  (\alpha=-\beta)\qquad\text{or}\qquad(\alpha=\beta\text{ and }\{-1,0,1\}\subset\D),
  \]
  then $X^{\D}(\alpha)=\Sigma_{\beta',\beta}(\Omega)$, where
  $\Omega=\mathcal{I}_{\frac{1}{\alpha'},\D}^{\circ}\cup\{0\}$.
\end{proposition}

The set $\mathcal{I}_{\frac{1}{\alpha'},\D}$ in the previous proposition is, in general,
defined as the set of all real numbers having a representation in the numeration system with base
$\frac{1}{\alpha'}$ and alphabet $\D$. However, if the alphabet
$\D=\{m,\ldots,0,1,\ldots,M\}$ is sufficiently large, i.e., $M-m>|\frac{1}{\alpha'}|-1$,
then
\begin{equation}\label{eq:inteval_I_gamma}
\mathcal{I}_{\frac{1}{\alpha'},\D} =
\begin{cases}
  \Big[\frac{m}{1-\alpha'},\frac{M}{1-\alpha'}\Big] & \text{if }\frac{1}{\alpha'}>1, \\[3mm]
  \Big[(M+m\frac{1}{\alpha'})\frac{\alpha'}{1-\alpha'^2},
      (M\frac{1}{\alpha'}+m)\frac{\alpha'}{1-\alpha'^2}\Big]
    &
    \text{if }\frac{1}{\alpha'}<-1. \\
\end{cases}
\end{equation}

Using this connection between spectra and cut-and-project sets we can decide whether a spectrum
is BDL or not.

\begin{theorem}\label{t:spectra}
  Let $\beta>1$ be a quadratic Pisot unit with conjugate $\beta'$ and
  $\D$ be an alphabet of consecutive integers such that $0\in\D$ and $\#\D>\beta$.
  Let $\alpha=-\beta$, or $\alpha = \beta$ and $\{-1,0,1\}\subset\D$.

  Then $X^{\D}(\alpha)$ is BDL if and only if
  \begin{equation}\label{eq:delicipodminka}
    \begin{array}{ll}
      \text{either} & \lfloor\beta\rfloor\ \text{divides}\ (\#\D-1) \text{ and } \beta'<0, \\[3pt]
      \text{or} & (\lfloor\beta\rfloor-1)\ \text{divides}\ (\#\D-1) \text{ and } \beta'>0.
    \end{array}
  \end{equation}
\end{theorem}

\begin{proof}
  Combining Propositions~\ref{t:MaPaPe} and~\ref{coro:Kesten} it is enough to study when
  \[
  \big|\mathcal{I}_{\frac{1}{\alpha'},\D}\big| \in \mathbb{Z} + \mathbb{Z}\beta.
  \]
  Let $\D = \{m,\ldots,0,\ldots,M\}$.  We will divide the proof into cases based on the minimal
  polynomial of $\beta$ and the sign of $\alpha$.

  \medskip

  \begin{enumerate}[I.]
  \item
    $\beta$ with minimal polynomial $x^2-\lfloor\beta\rfloor x-1$, i.e.,
    $\beta+\beta'=\lfloor\beta\rfloor$ and $\beta\beta'=-1$.

    \smallskip

    \begin{enumerate}[a)]
    \item
      Let $\alpha=\beta$, i.e., $\frac{1}{\alpha'}=\frac{1}{\beta'}<-1$. From~(\ref{eq:inteval_I_gamma})
      we get
      \begin{align*}
        \big|\mathcal{I}_{\frac{1}{\alpha'},\D}\big| & =
        \frac{\alpha'}{1-(\alpha')^2}\Big(M\frac{1}{\alpha'}+m-M-m\frac{1}{\alpha'}\Big) = \\[1mm]
         & = \frac{\alpha'}{1-(\alpha')^2}(M-m)\Big(\frac{1}{\alpha'}-1\Big) = \frac{M-m}{1+\alpha'}
        = \frac{M-m}{1+\beta'}.
      \end{align*}

    \item
      Let $\alpha=-\beta$, i.e., $\frac{1}{\alpha'}=-\frac{1}{\beta'}>1$. Then
      \[
         \big|\mathcal{I}_{\frac{1}{\alpha'},\D}\big|  = \frac{M-m}{1-\alpha'} =
         \frac{M-m}{1+\beta'}. 
      \]

    \end{enumerate}

    \noindent
    In both cases (by expanding the fraction by $\frac{1+\beta}{1+\beta}$) we get
    \[
    \frac{M-m}{1+\beta'} = \frac{(M-m)(1+\beta)}{1 + (\beta+\beta') + \beta\beta'} =
    (1+\beta)\frac{M-m}{\lfloor\beta\rfloor} = (1+\beta)\frac{\#\D-1}{\lfloor\beta\rfloor}.
    \]
    Therefore
    \[
    \big|\mathcal{I}_{\frac{1}{\alpha'},\D}\big| \in \mathbb{Z} + \mathbb{Z}\beta
    \quad\Longleftrightarrow\quad \lfloor\beta\rfloor\ \big|\ (\#\D-1).
    \]

    \medskip

  \item
    $\beta$ with minimal polynomial $x^2-\lceil\beta\rceil x+1$, i.e.,
    $\beta+\beta'=\lceil\beta\rceil$ and $\beta\beta'=1$.

    \smallskip

    \begin{enumerate}[a)]
    \item
      Let $\alpha=\beta$, i.e., $\frac{1}{\alpha'}=\frac{1}{\beta'}>1$. Then
      \[
      \big|\mathcal{I}_{\frac{1}{\alpha'},\D}\big| = \frac{M-m}{1-\alpha'} = \frac{M-m}{1-\beta'}.
      \]

    \item
      Let $\alpha=-\beta$, i.e., $\frac{1}{\alpha'}=-\frac{1}{\beta'}<-1$. Then
      \[
        \big|\mathcal{I}_{\frac{1}{\alpha'},\D}\big| = \frac{M-m}{1+\alpha'} = \frac{M-m}{1-\beta'}.
      \]

    \end{enumerate}

    \noindent
    In both cases
    \[
      \frac{M-m}{1-\beta'}
      = \frac{(M-m)(1-\beta)}{1 - (\beta+\beta') + \beta\beta'} =
      (\beta-1)\frac{M-m}{\lceil\beta\rceil-2} = (\beta-1)\frac{M-m}{\lfloor\beta\rfloor-1}.
    \]
    Therefore
    \[
      \big|\mathcal{I}_{\frac{1}{\alpha'},\D}\big| \in \mathbb{Z} + \mathbb{Z}\beta
      \quad\Longleftrightarrow\quad (\lfloor\beta\rfloor-1)\ \big|\ (\#\D-1).
    \]

  \end{enumerate}

\end{proof}


%

As it was shown in Proposition~\ref{t:MaPaPe}, the generalized spectrum $X^{\D}(\alpha)$ coincides
with a cut-and-project set, which has at most three distinct values of distances between consecutive
elements. Therefore the spectrum $X^{\D}(\alpha)$ is a geometric representation of a ternary or a
binary infinite word. For any infinite word $\bu$, balancedness implies that $\bu$ has a BDL
geometrical representation.  But existence of a BDL geometric representation does not force
balancedness of $\bu$. For the spectrum $X^{\D}(\alpha)$ of quadratic Pisot units these two
properties are equivalent.

\begin{corollary}\label{coro:ekvivalenceSpektra}
  Let $\beta$, $\alpha$ and $\D$ satisfy the assumptions of Theorem~\ref{t:spectra}.  If
  $X^{\D}(\alpha)$ is bounded distance equivalent to a lattice $\xi \mathbb{Z}$, then $\xi
  = \frac{\beta-\beta'}{\#\D-1}(1-|\beta'|)$ and the infinite word coding $X^{\D}(\alpha)$ is
  balanced.
\end{corollary}
\begin{proof}
  By Proposition~\ref{t:MaPaPe}, $X^{\D}(\alpha)=\Sigma_{\beta',\beta}(\Omega)$. Using
  Corollary~\ref{coro:ekvivalenceCAP} with $\varepsilon = \beta'$ and $\eta = \beta$, we get that
  $X^{\D}(\alpha)$ is BDL to $\xi \mathbb{Z}$ with $\xi = \frac{\beta - \beta'}{|\Omega|}$
  and the word coding $X^{\D}(\alpha)$ is balanced.  To determine $\xi$ we discuss two
  cases:

  If $\beta' <0$, then by proof of Theorem~\ref{t:spectra}, the length of the acceptance interval is
  $|\Omega| = (1+\beta) \frac{\#\D-1}{\floor{\beta}}$.  As $\floor{\beta} = \beta+\beta' =
  (1+\beta)(1+\beta')$ we deduce
  \[
  \xi=   \frac{(\beta-\beta')\floor{\beta}}{(\#\D-1)(1+\beta)}  =  \frac{\beta-\beta'}{\#\D-1}(1+\beta')\,.
  \]

  If $\beta' >0$, then by proof of Theorem~\ref{t:spectra}, the length of the acceptance interval is
  $|\Omega| = (\beta-1) \frac{\#\D-1}{\floor{\beta}-1}$.  As $\floor{\beta}-1 = (\beta-1)(1-\beta')$
  we deduce
  \[
  \xi=   \frac{(\beta-\beta')(\floor{\beta}-1)}{(\#\D-1)(\beta-1)}  =  \frac{\beta-\beta'}{\#\D-1}(1-\beta')\,.
  \]
\end{proof}

\begin{example} \label{ex:examplecontinue2}
  The most simple example of the generalized spectrum in the sense of Proposition~\ref{t:MaPaPe} is
  the spectrum of $-\tau$ with the alphabet $\D=\{0,1\}$.  From Proposition~\ref{t:MaPaPe} we see
  that $X^{\{0,1\}}(-\tau) = \Sigma_{\tau',\tau}[0,\tau^2)$. Using relation~\eqref{eq:transformaceCAP}
  and the unimodular matrix $\left(\begin{smallmatrix} 0 & -1 \\ 1 & 2\end{smallmatrix}\!\right)$
  we can write that
  \[
  \Sigma_{{\tau'}^2,\tau^2}[0,1) = \tau^2 \Sigma_{\tau',\tau}[0,\tau^2).
  \]
  We have thus identified the spectrum with a scaled copy of the cut-and-project set from
  Example~\ref{ExampleContinue}. From there we derive that the distances in
  $X^{\{0,1\}}(-\tau)=\tau^{-2}\Sigma_{{\tau'}^2,\tau^2}[0,1)$ take values $1$ and $\tau^{-1}$.

  It is known~\cite{DoMaVa} that the spectrum $X^{\{0,1\}}(-\tau)$ coincides with the set
  $\Z_{-\beta}$ of $(-\tau)$-integers and that the infinite word over the alphabet $\A=\{A,B\}$,
  $l_A=1$, $l_B=\tau^{-1}$ coding the ordering of distances in $\Z_{-\beta}$ is a fixed point of the
  substitution $A\mapsto AAB$, $B\mapsto AB$, with the initial pair $A|B$. This identifies the
  spectrum $X^{\{0,1\}}(-\tau)$ with a geometric representation
  of the fixed point presented in Example~\ref{ex:selfsimgeomrep}. There we have already calculated
  frequencies of the letters $ \rho_A = \tau^{-1}$ and $ \rho_B = \tau^{-2}$.  Since the incidence
  matrix of the substitution has eigenvalues $\tau^2 >1$ and $\tau^{-2}< 1$, Theorem~\ref{t:Adam}
  implies that the binary word is balanced and therefore by Proposition~\ref{p:anygeomBDL} its
  geometrical representation with the length $ \ell_A = 1$ and $ \ell_B = \tau^{-1}$ is bounded
  distance to the lattice $\xi \mathbb{Z}$, with
  $\xi = \ell_A \rho_A +\ell_B \rho_B = \tau^{-1}+\tau^{-3}$.
  Let us note, that the same value $\xi$ can be computed also using Corollary~\ref{coro:ekvivalenceSpektra}.

  In Figure~\ref{fig:bdl_repr_(-tau)}, several elements of the spectrum $X^{\{0,1\}}(-\tau)$ are
  drawn and marked with their $(-\tau)$-representation. For example, $101{\scriptstyle\bullet} =
  (-\tau)^2+(-\tau)^0$.  The figure also shows the bijection between the spectrum and the lattice,
  to which it is bounded distance equivalent.
\end{example}

\begin{figure}[!htp]

  \centering

  \begin{tikzpicture}[x=1.5cm,y=1.5cm]

    \draw (-5.1246,0) node [left, xshift = -10] {$\eta\Z$:} -- (3.618,0);
    \draw (-5.1246,1) node [left, xshift = -10] {$X^{\{0,1\}}(-\tau)$:} -- (3.618,1);

    \draw (-4.854,0.95) -- (-4.854,1.05) node [above left, yshift = -3, rotate = -45]
          {$1011\scriptstyle\bullet$};
    \draw (-4.236,0.95) -- (-4.236,1.05) node [above left, yshift = -3, rotate = -45]
          {$1000\scriptstyle\bullet$};
    \draw (-3.236,0.95) -- (-3.236,1.05) node [above left, yshift = -3, rotate = -45]
          {$1110\scriptstyle\bullet$};
    \draw (-2.236,0.95) -- (-2.236,1.05) node [above left, yshift = -3, rotate = -45]
          {$1111\scriptstyle\bullet$};
    \draw (-1.618,0.95) -- (-1.618,1.05) node [above left, yshift = -3, rotate = -45]
          {$10\scriptstyle\bullet$};
    \draw (-0.618,0.95) -- (-0.618,1.05) node [above left, yshift = -3, rotate = -45]
          {$11\scriptstyle\bullet$};
    \draw (0,0.95) -- (0,1.05) node [above left, yshift = -3, rotate = -45]
          {$0\scriptstyle\bullet$};
    \draw (1,0.95) -- (1,1.05) node [above left, yshift = -3, rotate = -45]
          {$1\scriptstyle\bullet$};
    \draw (2,0.95) -- (2,1.05) node [above left, yshift = -3, rotate = -45]
          {$111\scriptstyle\bullet$};
    \draw (2.618,0.95) -- (2.618,1.05) node [above left, yshift = -3, rotate = -45]
          {$100\scriptstyle\bullet$};
    \draw (3.618,0.95) -- (3.618,1.05) node [above left, yshift = -3, rotate = -45]
          {$101\scriptstyle\bullet$};

    \draw (-5.1246,-0.05) node [anchor = north] {$-6\eta$} -- (-5.1246,0.05);
    \draw (-4.2705,-0.05) node [anchor = north] {$-5\eta$} -- (-4.2705,0.05);
    \draw (-3.4164,-0.05) node [anchor = north] {$-4\eta$} -- (-3.4164,0.05);
    \draw (-2.5623,-0.05) node [anchor = north] {$-3\eta$} -- (-2.5623,0.05);
    \draw (-1.7082,-0.05) node [anchor = north] {$-2\eta$} -- (-1.7082,0.05);
    \draw (-0.8541,-0.05) node [anchor = north] {$-\eta$} -- (-0.8541,0.05);
    \draw (0.0000,-0.05) node [anchor = north] {$0$} -- (0.0000,0.05);
    \draw (0.8541,-0.05) node [anchor = north] {$\eta$} -- (0.8541,0.05);
    \draw (1.7082,-0.05) node [anchor = north] {$2\eta$} -- (1.7082,0.05);
    \draw (2.5623,-0.05) node [anchor = north] {$3\eta$} -- (2.5623,0.05);
    \draw (3.4164,-0.05) node [anchor = north] {$4\eta$} -- (3.4164,0.05);

    \draw [gray,<->] (-5.1246,0.1) -- (-4.854,0.9);
    \draw [gray,<->] (-4.2705,0.1) -- (-4.236,0.9);
    \draw [gray,<->] (-3.4164,0.1) -- (-3.236,0.9);
    \draw [gray,<->] (-2.5623,0.1) -- (-2.236,0.9);
    \draw [gray,<->] (-1.7082,0.1) -- (-1.618,0.9);
    \draw [gray,<->] (-0.8541,0.1) -- (-0.618,0.9);
    \draw [gray,<->] (0,0.1) -- (0,0.9);
    \draw [gray,<->] (0.8541,0.1) -- (1,0.9);
    \draw [gray,<->] (1.7082,0.1) -- (2,0.9);
    \draw [gray,<->] (2.5623,0.1) -- (2.618,0.9);
    \draw [gray,<->] (3.4164,0.1) -- (3.618,0.9);

  \end{tikzpicture}

  \caption{Illustration of bounded distance equivalence between elements of the spectra
    $X^{\{0,1\}}(-\tau)$ and points of the lattice $\eta\Z$.}\label{fig:bdl_repr_(-tau)}
\end{figure}
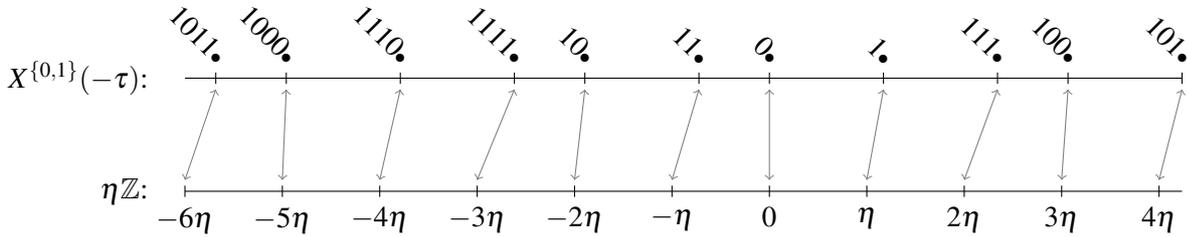




\section{Comments}

One-dimensional BDL sequences can be used to construct structures having an average lattice in
higher dimension. In particular, if $\Lambda_1,\Lambda_2\subset\R$ are BDL sequences, then the set
$\Lambda_1\vec{u}+\Lambda_2\vec{v}=\{x\vec{u}+y\vec{v}: x\in\Lambda_1,y\in\Lambda_2\}\subset\R^2$ is a BDL set for
any pair of linearly independent vectors $\vec{u},\vec{v}\in\R^2$.  In Figure~\ref{fig:2d}, the
Fibonacci grid is displayed together with its average lattice. The generating vectors
$\vec{u},\vec{v}$ form the angle $2\pi/5$.

\begin{figure}[!htp]

  \centering

  \begin{tikzpicture}[x=1.2cm,y=1.2cm]

    \clip (-6.75,-5.25) rectangle (3.5, 2.9);


    \draw [lightgray] (-6.7082,-4.8737)
        -- ++(0:0.8542) coordinate (mh-1)
        -- ++(0:0.8542) coordinate (mh-2)
        -- ++(0:0.8542) coordinate (mh-3)
        -- ++(0:0.8542) coordinate (mh-4)
        -- ++(0:0.8542) coordinate (mh-5)
        -- ++(0:0.8542) coordinate (mh-6)
        -- ++(0:0.8542) coordinate (mh-7)
        -- ++(0:0.8542) coordinate (mh-8)
        -- ++(0:0.8542) coordinate (mh-9);

    \foreach \i in {1,...,9} {
      \draw [lightgray] (mh-\i) -- ++(72:7.6868);
    }

    \draw [lightgray] (-6.7082,-4.8737)
        -- ++(72:0.8542) coordinate (mv-1)
        -- ++(72:0.8542) coordinate (mv-2)
        -- ++(72:0.8542) coordinate (mv-3)
        -- ++(72:0.8542) coordinate (mv-4)
        -- ++(72:0.8542) coordinate (mv-5)
        -- ++(72:0.8542) coordinate (mv-6)
        -- ++(72:0.8542) coordinate (mv-7)
        -- ++(72:0.8542) coordinate (mv-8)
        -- ++(72:0.8542) coordinate (mv-9);

    \foreach \i in {1,...,9} {
      \draw [lightgray] (mv-\i) -- ++(0:7.6868);
    }


    \path (-6.3539,-4.6164) coordinate (v-0) 
         -- ++(72:0.618) coordinate (v-1) 
         -- ++(72:1) coordinate (v-2) 
         -- ++(72:1) coordinate (v-3) 
         -- ++(72:0.618) coordinate (v-4) 
         -- ++(72:1) coordinate (v-5) 
         -- ++(72:0.618) coordinate (v-6) 
         -- ++(72:1) coordinate (v-7) 
         -- ++(72:1) coordinate (v-8) 
         -- ++(72:0.618) coordinate (v-9) 
         -- ++(72:1) coordinate(v-10); 

    \path (-6.3539,-4.6164) coordinate (h-0)
         -- ++(0:0.618) coordinate (h-1)
         -- ++(0:1) coordinate (h-2)
         -- ++(0:1) coordinate (h-3)
         -- ++(0:0.618) coordinate (h-4)
         -- ++(0:1) coordinate (h-5)
         -- ++(0:0.618) coordinate (h-6)
         -- ++(0:1) coordinate (h-7)
         -- ++(0:1) coordinate (h-8)
         -- ++(0:0.618) coordinate (h-9)
         -- ++(0:1) coordinate(h-10);

    \foreach \i in {0,1,...,10} {
      \path [name path global = horiz-\i] (v-\i) -- 
      ++(0:8.472);
      \path [name path global = vert-\i] (h-\i) -- 
      ++(72:8.472);
     }

    \foreach \sx in {0,1,...,9} {
      \foreach \sy in {0,1,...,9} {
        \fill [black, name intersections={of={horiz-\sx} and {vert-\sy}}] (intersection-1) circle (2pt);
      }
    }

    \draw (0,0) node [anchor = north west] {(0,0)};
    \draw (0:0) -- ++(0:3);
    \draw (0:0) -- ++(180:5.5);
    \draw (0:0) -- ++(72:3);
    \draw (0:0) -- ++(252:5.5);

\end{tikzpicture}

  \caption{The Fibonacci grid and its average lattice.}\label{fig:2d}

\end{figure}
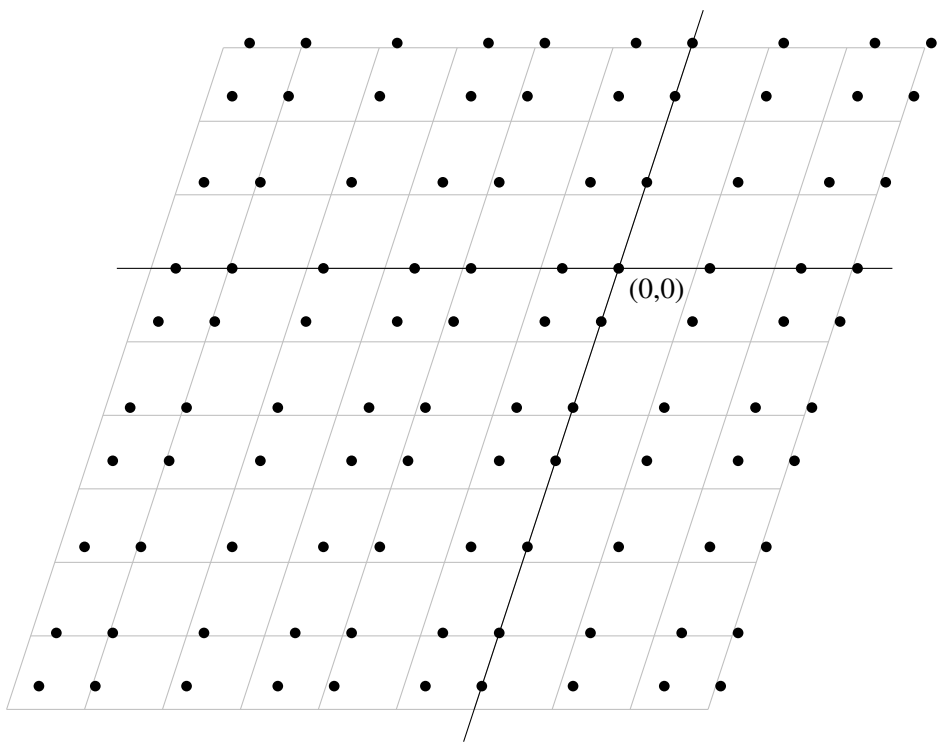

The Cartesian product of one-dimensional BDL sets is a special case of higher-dimensional structures
bounded distance equivalent to a lattice.  The spectra of Pisot-cyclotomic numbers provide a 2D
structure having moreover non-crystallographic symmetries. In~\cite{HaMaVa}  the spectra of
quadratic Pisot-cyclotomic numbers are identified with cut-and-project sets arising from dimension
4. Their BDL property is yet to be explored.

%
%
%
%
%


\section*{Acknowledgments}


This work was supported by the project CZ.02.1.01/0.0/0.0/16\_019/0000778 from European Regional Development Fund.



\printbibliography




\end{document}